\documentclass[1 [leqno,11pt]{amsart}
\usepackage{amssymb, amsmath,amsmath,latexsym,amssymb,amsfonts,amsbsy, amsthm,mathtools,graphicx,color}

\renewcommand{\theequation}{\thesection.\arabic{equation}}
\numberwithin{equation}{section}

\allowdisplaybreaks

\let\al=\alpha

\let\d=\delta
\let\e=\varepsilon

\let\la=\lambda

\let\f=\frac

\let\Om=\Omega

\let\tri=\triangle
\let\na=\nabla

\let\pa=\partial

\def\eqdefa{\buildrel\hbox{\footnotesize def}\over =}
\def\cA{{\cal A}}

\def\mS{\mathbb{S}}
\def\mD{\mathbb{D}}
\def\cA{\mathcal{A}}

\def\brho{\overline{\rho}}

\def\dv{\mbox{div}}
\def\dive{\mathop{\rm div}\nolimits}

\newcommand{\beq}{\begin{equation}}
\newcommand{\eeq}{\end{equation}}
\newcommand{\ben}{\begin{eqnarray}}
\newcommand{\een}{\end{eqnarray}}
\newcommand{\beno}{\begin{eqnarray*}}
\newcommand{\eeno}{\end{eqnarray*}}

\newtheorem{theorem}{Theorem}[section]

\newtheorem{lemma}[theorem]{Lemma}
\newtheorem{proposition}[theorem]{Proposition}

\newtheorem{remark}[theorem]{Remark}

\begin{document}
\title[Local well-posedness of 3-D CNS with  vacuum]
{Local well-posedness of the vacuum free boundary of 3-D compressible Navier-Stokes equations}

\author[G. Gui]{Guilong Gui}
\address{Center for Nonlinear Studies, School of Mathematics\\ Northwest University\\
Xi¡¯an 710069, China}
\email{glgui@amss.ac.cn}

\author[C. Wang]{Chao Wang}
\address{School of Mathematical Sciences\\ Peking University\\ Beijing 100871,China}
\email{wangchao@math.pku.edu.cn}

\author[Y. Wang]{Yuxi Wang}
\address{School of Mathematical Sciences\\ Peking University\\ Beijing 100871,China}
\email{wangyuxi0422@pku.edu.cn}

\date{\today}

\maketitle

\begin{abstract}
In this paper, we consider the 3-D motion of viscous gas with the vacuum free boundary. We use the conormal derivative to establish local well-posedness of this system. One of important advantages in the paper is that we do not need any strong compatibility conditions on the initial data in terms of the acceleration.

\end{abstract}

\renewcommand{\theequation}{\thesection.\arabic{equation}}
\setcounter{equation}{0}

\section{Introduction}
\subsection{Formulation in Eulerian Coordinates}

In the paper, we consider a 3-D viscous compressible fluid in a moving domain $\Om(t)$ with an upper free surface $\Gamma(t)$ and a fixed bottom $\Gamma_b$. This model can be expressed by the 3-D compressible Navier-Stokes equations(CNS)
\begin{equation}\label{eq:CNS-Free}
\begin{cases}
&\pa_t \rho+\na\cdot (\rho\,u)=0\quad\mbox{in}\quad \Om(t),\\
&\rho(\pa_t u+u\cdot \na u)+\na p-\na\cdot \mathbb{S}(u)=0\quad\mbox{in}\quad \Om(t),\\
&\rho>0  \quad\mbox{in}\quad \Om(t),\quad \rho=0 \quad\mbox{on}\quad \Gamma(t),\\
&\mathcal{V}(\Gamma (t))=u\cdot n \quad\mbox{on}\quad \Gamma(t),\\
&(\mathbb{S}(u)-p\, \mathbb{I})n=0 \quad\mbox{on}\quad \Gamma(t),\\
&u|_{\Gamma_{b}}=0 \quad\mbox{on}\quad \Gamma_{b},\\
&(\rho,u)|_{t=0}=(\rho_0,u_0) \quad\mbox{in}\quad \Om(0),\quad\Om(0)=\Om_0,
\end{cases}
\end{equation}
where $\mathcal{V}(\Gamma (t))$ denotes the normal velocity of the free surface $\Gamma(t)$, and $n= n(t)$ is the exterior unit normal vector of $\Gamma(t)$, the vector-field $u$ denotes the Eulerian velocity field, $\rho$ is the density of the fluid, and $p=p(\rho)$ denotes the pressure function. The stress tensor $\mathbb{S}(u)$ is defined by $\mathbb{S}(u)=\mu\mathbb{D}(u)+\la (\na \cdot u)\mathbb{I}$, where the strain tensor $\mathbb{D}(u)=\na u+\na u^{T}$ and dynamic viscosity $\mu$ and bulk viscosity $\nu$ are constants which satisfy the following relationship
 \begin{align}\label{assum: mu,la}
 \mu>0,~\lambda+\frac{2}{3}\mu\geq 0.
 \end{align}
  The deviatoric (trace-free) part of the strain tensor $\mathbb{D}(u)$ is then $\mathbb{D}^0(u)=\mathbb{D}(u)-\frac{2}{3}\dive\, u\, \mathbb{I}$. The viscous stress tensor in fluid is then given by $\mathbb{S}(u)=\mu\,\mathbb{D}^0(u)+(\lambda+\frac{2}{3}\mu) (\na \cdot u)\, \mathbb{I}$.
 Moreover, the pressure obeys the $\gamma$-law: $p(\rho)=K\,\rho^{\gamma}$, where $K$ is an entropy constant and $\gamma>1$ is the adiabatic gas exponent.

Equation $\eqref{eq:CNS-Free}_1$ is the conservation of mass; Equation $\eqref{eq:CNS-Free}_2$ means the
momentum conserved; the boundary condition $\eqref{eq:CNS-Free}_3$ states that
the pressure (and hence the density function) vanishes along the moving boundary
$\Gamma(t)$, which indicates that the vacuum state appears on the boundary $\Gamma(t)$; the kinematic boundary condition $\eqref{eq:CNS-Free}_4$ states that the vacuum boundary $\Gamma(t)$ is moving with speed equal to the normal component of the fluid velocity; $\eqref{eq:CNS-Free}_5$ means the fluid satisfies the kinetic boundary condition on the free boundary, $\eqref{eq:CNS-Free}_6$ denotes the fluid is no-slip, no-penetrated on the fixed bottom boundary, and $\eqref{eq:CNS-Free}_{7}$ are the initial conditions for the density, velocity, and domain.

In the paper, we assume the bottom $\Gamma_{b}=\{y_3=b(y_h)\}$, and the moving domain $\Om(t)$ is horizontal periodic by setting $\mathbb{T}^2_{y_h}$ with $y_h:=(y_1, y_2)^T$ for $\mathbb{T}=\mathbb{R}/\mathbb{Z}$.

\bigskip

\subsection{Known results}

Whether or not the appearance of vacuum state is related to the regularity of the solution to the compressible Navier-Stokes equations. Even if there is no vacuum in initial data, it cannot guarantee that vacuum state will be not generated in finite time in high-dimensional system. Whence initial data is close to a non-vacuum equilibrium in some functional space, Matsumura and Nishida\cite{MN1,MN2} proved global well-posedness of strong solutions to the 3-D CNS. Moreover, for the one dimensional case, Hoff and Smoller \cite{HoSo2001} proved that if the vacuum is not included at the beginning, no vacuum will occur in the future. Hoff and Serre \cite{HoSe1991} showed some physical weak solution does not have to depend continuously on their initial data when vacuum occurs.

When the initial density may vanish in open sets or on the (part of) boundary of the domain, the flow density may contain a vacuum, the equation of velocity becomes a strong degenerate hyperbolic-parabolic system and the degeneracy is one of major difficulties in study of regularity and the solution's behavior, which is completely different from the non-vacuum case. For the existence of solutions for arbitrary data (the far field density is vacuum, that is, $\rho(t, x)\rightarrow 0$ as $x\rightarrow \infty$), the major breakthrough is due to Lions \cite{LION1998} (also see \cite{Feireisl2001, Jiang2001, Hoff2005}), where he obtains global existence of weak solutions, defined as solutions with finite energy with suitable $\gamma$. Recently, Li-Xin \cite{LiXin2015} and Vasseur-Yu \cite{Vasseur2016} independently studied global existence of weak solutions of CNS whence the viscosities depend on the density and satisfy the Bresch-Desjardins relation \cite{BDE2003}.
Yet little is known on the structure of such weak solutions except for the case that
some additional assumptions are added (see \cite{HS2008} for example). Indeed, the works of Xin etc. \cite{Xin, LWZ} showed that the homogeneous Sobolev space is as crucial as studying the well-posedness for the Cauchy problem of compressible Navier-Stokes equations in the presence of a vacuum at far fields even locally in time. Adding some compatible condition on initial data,  Cho and Kim \cite{CK1} develop local well-posedness for strong solutions. Moreover, if initial energy is small, Huang, Li and Xin \cite{HLX} showed the global existence of classical solutions but with large oscillations to CNS.

Physically, the vacuum problem appears extensively in the fundamental free boundary hydrodynamical setting: for instance, the evolving boundary of a viscous gaseous star, formation of shock waves, vortex sheets, as well as phase transitions.

For free boundary problem of the multi-dimensional Navier-Stokes equations with non-vacuum state, there are many results concerning its local and global strong solutions, one may refer to \cite{Zadrzynska2001,Zajaczkowski1999} and references therein.

But when the vacuum (in particular, the physical vacuum \cite{LY}) appears, the system becomes much harder. To understand the difficulty of the vacuum, we introduce the sound speed $c:=\sqrt{p'(\rho)} (=\sqrt{K\gamma} \rho^{\frac{\gamma-1}{2}}$  for polytropic gases) of the gas or fluid to describe the behavior of the smoothness of the density connecting to vacuum boundary. A vacuum boundary $\Gamma(t)$ is called physical vacuum if there holds
\begin{equation}\label{con:physical}
-\infty<\f{\pa c^2}{\pa n}<0
\end{equation}
near the boundary $\Gamma(t)$, where $n$ is the outward unit
normal to the free surface. The physical vacuum condition \eqref{con:physical} implies the pressure (or the enthalpy $c^2$) accelerates the boundary in the normal direction. Thus, the initial physical vacuum condition \eqref{con:physical} is equivalent to the requirement that
\begin{equation}\label{con:physical-2}
-\infty<\partial_n (\rho_0^{\gamma-1})<0 \quad \text{on} \quad \Gamma(0)
\end{equation}
which means that $\rho^{\gamma-1}_0(x)\sim dist(x,\Gamma (0))$, in other words, the initial sound speed $c_0$ is only $C^{\frac{1}{2}}$-H\"{o}lder continuous near the interface $\Gamma(0)$.

Due to lack of sufficient smoothness of the enthalpy $c^2$ at the vacuum boundary, a rigorous understanding of the existence of physical vacuum states in compressible fluid dynamics has been a challenging problem, especially in multi-dimensional cases.

Recently, the local well-posedness theory for compressible Euler system with physical vacuum singularity was established in \cite{CS2, JM1,JM2}, and also global existence of smooth solutions for the physical vacuum free boundary problem of the 3-D spherically symmetric compressible Euler equations with damping was showed in \cite{LZ2016}. And more recently, Hadzic and Jang \cite{HJ2018} proved global nonlinear stability of the affine solutions to the compressible Euler system with physical vacuum, and Guo, Hadzic, and Jang \cite{GHJ2018} constructed an infinite dimensional family of collapsing solutions to the Euler-Poisson system whose density is in general space inhomogeneous and undergoes gravitational blowup along a prescribed space-time surface, with continuous mass absorption at the origin.

The study of vacuum is important in understanding viscous surface flows \cite{LXY2000}. Very little is rigorously known about well-posedness theories available about  free boundary problems of CNS with physical vacuum boundary. For 1-D problem, global regularity for weak solutions to the vacuum free boundary problem of CNS was obtained in \cite{LXY2000}, which is further generalized by Zeng \cite{Zeng2015} which established the strong solutions. For the multidimensional case, regularity results related to spherically symmetric motions. Guo, Li and Xin \cite{GLX} obtain a global weak solution to the problem with spherically symmetric motions and a jump density connects to vacuum. Later Liu\cite{Liu2018} gives the existence of global solutions with small energy in spherically symmetric motions with the density connected to vacuum continuously or discontinuously. Anyway, almost all the well-posedness results require additional strongly singular compatibility conditions on initial data in terms of the acceleration for gaining more regularities of the velocity. Some related works can refer to \cite{CK,OM, Li2016,Guo2012b,Yeung2009,Jang-1,FZ,FZ1,LY, LXY2000,YT2006,LXZ2016} and references therein.

The purpose of this paper is to establish the local well-posedness of the 3-D compressible Navier-Stokes equations \eqref{eq:CNS-Free} with physical vacuum boundary condition without any compatibility conditions, more precisely, we do not need any initial condition on the material derivative $D_t u$ or its derivatives. For simplicity, we set $\gamma=2$ and $K=1$ in this paper.

As mentioned above, the main difficulty in obtaining regularity for the vacuum free boundary problem \eqref{eq:CNS-Free} lies in the degeneracy of the system near vacuum boundaries. In order to solve the system \eqref{eq:CNS-Free}, the first idea is that we use Lagrangian coordinates to transform it to a system with fixed domain. One of advantage of Lagrangian coordinates is that the density $\rho$ is solved directly by initial data and we only focus on the equation of velocity with coefficients related to Lagrangian coordinates.

The second and also key idea in our paper is that we use the conormal derivatives to obtain the high-order regularity. Because the density vanishes on the boundary, we can not close the energy estimates if we directly take normal derivatives to the system. So another choose is to take time derivatives in \cite{CS2, JM2} solving the compressible Euler equations with the physical vacuum, where high-order enough time-derivative estimates as long as spatial-derivative estimates allow us to close the energy estimates and then get the local-in-time existence of the strong solution of the Euler system. This high-order energy estimate in it is reasonable since the pressure term may cancel the singularity near the vacuum boundary when consider compatibility conditions on initial data in terms of the acceleration and its derivatives. However, this method may not work for the Navier-Stokes system \eqref{eq:CNS-Free} with constant viscosity coefficients. In fact, a strong singular compatibility conditions on initial data in terms of the acceleration and its derivatives will appear in it when we consider the high-order energy estimate, which is mainly due to the non-degenerate of the viscosity, but it seems very hard to find such kind of initial data satisfying these compatibility conditions. In order to get rid of this difficulty, our strategy is that we use conormal Sobolev space introduced in \cite{MR} to get the tangential regularity. Based on that, we multiply $\pa_t v$ on the both sides of equations of $v$ to get the estimates of $\rho^{\f12}\pa_t v$ which implies the two-order derivative on the normal direction. Form this, together with high-order tangential derivatives estimates, we get the $W^{1,\infty}$ estimates of $v$ and its conormal derivatives, which in turn guarantees the propagation of conormal regularities of the velocity.

\subsection{Derivation of the system in Lagrangian coordinate and main result}

In this paper, we consider the case that the upper boundary does not touch the bottom which means that
\begin{equation*}
dist(\Gamma(0), \Gamma_b)>0.
\end{equation*}
Take $\Om=\{x \in \mathbb{T}^2\times \mathbb{R}|\quad 0 <x_3<1\}$ as the domain of equilibrium.
Let $\eta(t,x)$ be the position of the gas particle $x$ at time $t$ so that
\begin{align}\label{eq:flow map}
\left\{
\begin{aligned}
&\pa_t \eta(t,x)=u(t,\eta(t,x)) \quad\mbox{for}\quad t>0,\\
&\eta(0,x)=\eta_0(x) \quad\mbox{in}\quad {\Omega}.
\end{aligned}
\right.
\end{align}
Here $\eta_0$ is a diffeomorphism from ${\Om}$ to the initial moving domain $\Om(0)$ which satisfies that $\Gamma(0)=\eta_0(\{x_3=1\})$ and $\Gamma_{b}=\eta_0(\{x_3=0\})$.  It is easy to construct a invertible transform $\eta_0$ which satisfies that
\beno
det D\eta_0 >0.
\eeno

Due to \eqref{eq:flow map}, we introduce the displacement $\xi(t, x)\eqdefa\eta(t, x)-x$ which satisfies the following ODE
\begin{equation}\label{eq:flow map-1}
\begin{cases}
&\pa_t \xi(t,x)=u(t, x+\xi(t,x)) \quad\mbox{for}\quad t>0,\\
&\xi(0,x)=\xi_0(x):=\eta_0(x)-x \quad\mbox{in}\quad {\Omega}.
\end{cases}
\end{equation}

Now we define the following Lagrangian quantities:
\begin{equation*}
\begin{split}
&v(t,x):=u(t,\eta(t,x)),\quad\, f(t,x):=\rho(t,\eta(t,x)),\\
&\mathcal{A}:=[D\eta]^{-1},\quad J:=det D\eta, \quad \mathcal{N}:=J\mathcal{A}\,e_3.
\end{split}
\end{equation*}

Then, the system \eqref{eq:CNS-Free} becomes
\begin{align}\label{eq:CNS-eqns-1}
\left\{
\begin{aligned}
&\partial_t\xi=v \quad\mbox{in}\quad \Om,\\
&\pa_t f+f\na_{\cA}\cdot v=0 \quad\mbox{in}\quad \Om,\\
&f\pa_t v+\na_{\cA} (f^2)-\na_{\cA}\cdot \mS_{\cA} v=0 \quad\mbox{in}\quad \Om
\end{aligned}
\right.
\end{align}
with boundary conditions
\begin{equation}\label{eq:CNS-eqns-2}
\begin{cases}
&f=0  \quad\mbox{on}\quad \Gamma,\\
&\mS_{\cA} (v)\, \mathcal{N}=0,\quad\mbox{on}\quad \Gamma,\\
&v|_{x_3=0}=0
\end{cases}
\end{equation}
and initial data
\begin{align}\label{eq:CNS-eqns-3}
&(\xi, f,v)|_{t=0}=(\xi_0, \rho_0,u_0).
\end{align}

One may readily check that
\begin{align}
\pa_t J=J\na_{\cA}\cdot v.
\end{align}
Combining with the equation of $f$ in \eqref{eq:CNS-eqns-1}, we find that
\begin{align*}
\pa_t(fJ)=J\pa_t f+f\pa_t J =-Jf\,\na_{\cA}\cdot v+fJ\na_{\cA}\cdot v =0.
\end{align*}
Then we get
\begin{align*}
Jf(t,x)=(Jf)(0,x)=\det (D\eta_0)\rho_0(\eta_0),
\end{align*}
where $\rho_0$ is a given initial density function. We are interested in the initial $\rho_0$ satisfying
\begin{align}
&\rho_0(\eta_0)\det (D\eta_0)=\overline{\rho}(x),\label{assum: brho1}\\
&C^{-1}\,d(x)\leq \brho(x)\leq C\,d(x),\label{assum: brho3}\\
&|\na \brho|\leq C,~|\brho^{-1}\nabla_h^k \brho|\leq C_k\label{assum: brho2}
\end{align}
with some given function $\brho(x)$ ($x\in \Omega$), for any $k \in \mathbb{N}$ with $\nabla_h=(\pa_1,\pa_2)$, where $d(x)$ is the distance function to the boundary $\{x_3=1\}$.

Thus, we have
\begin{equation}\label{eq: brho}
\begin{split}
Jf=\overline{\rho}(x),
\end{split}
\end{equation}
which implies that
\begin{equation*}\label{f-q-1}
\begin{split}
   f= J^{-1}\,\overline{\rho},\quad q=f^2=J^{-2}\,\overline{\rho}^{2}.
\end{split}
\end{equation*}

Multiplying $J$ on the both side of equation $v$, we deduce the equivalent form of the system \eqref{eq:CNS-eqns-1}-\eqref{eq:CNS-eqns-3} as follows
\begin{equation}\label{eq:CNS1}
\begin{cases}
&\partial_t\xi=v \quad\mbox{in}\quad \Om,\\
&\brho\pa_t v+\na_{J\cA}(J^{-2}\brho^2)-\na_{J\cA}\cdot \mS_{\cA}v=0 \quad\mbox{in}\quad \Om,\\
&\mS_{\cA} (v)\,\mathcal{N}=0,\quad\mbox{on}\quad \Gamma,\\
&v|_{x_3=0}=0,\\
&\xi|_{t=0}=\xi_0,\quad \, v|_{t=0}=v_0  \quad\mbox{in}\quad \Om.
\end{cases}
\end{equation}

\medskip

Next, we give some useful equations which we often use. Since $\cA[D\eta]=I,$ one obtains that
\begin{align}\label{equ:d-A}
\pa_t \cA_i^k=-\cA^{k}_{r}\pa_s v^r\cA_i^s,\quad \pa_l \cA_i^k=-\cA^k_r\pa_l\pa_s\eta^r\cA_i^s.
\end{align}
Differentiating the Jacobian determinant, one can see that
\begin{align}\label{equ:d-J}
\pa_t J=J\cA^{s}_{r}\pa_s v^r,\quad \pa_l J=J\cA^s_r\pa_l\pa_s\eta^r.
\end{align}
Moreover, the following Piola identity we heavily used:
\begin{align}\label{Piola-identity}
\pa_j(J\cA^j_i)=0,
\end{align}
for any $i=1,2,3.$

\subsection{Main results} Before we state our main results, we give some definitions of functional spaces. First, define the operators:
\begin{align}
Z_1\eqdefa\pa_{1},\quad Z_2\eqdefa\pa_{2},\quad Z_3\eqdefa\brho\pa_{3}.
\end{align}
Using $Z^m$ to denote $Z_3^{m_2}Z_h^{m_1}=Z_3^{m_2}Z_1^{m_{11}}Z_2^{m_{12}}$ and $|m|$ to denote $|m|=|m_1|+m_2=m_{11}+m_{12}+m_2.$ Moreover, we use $Z_3^{m_2}$ to denote $\brho^{m_2}\pa_{3}^{m_2}.$ By \eqref{assum: brho1}-\eqref{assum: brho2}, it is easy to see
\begin{align*}
[\pa_3,Z^m]\sim Z^{m-1}\pa_3,~[Z_h,Z_3]\sim Z_3.
\end{align*}

We recall the following conormal Sobolev space introduced by \cite{MR}.
\begin{align*}
& \|v\|_{X^N_{\al}}^2:=\sum_{|m|=0}^N\d^{2|m|} \|\brho^{\al} Z^m v\|_{L^2}^2,\quad \|v\|_{\dot{X}^N_\al}^2:=\sum_{ |m|=1}^N\d^{2|m|}\|\brho^{\al} Z^m v\|_{L^2}^2,
\end{align*}
where $\d $ is a small constant which determined later and $\al\in\mathbb{R}$. In particular, when $\al=0$, we the spaces $X^N_{\alpha}$ and $\dot{X}^N_{\alpha}$ will be denoted by $X^N$ and $\dot{X}^N$ respectively for simplicity.
\medskip

For $T>0$, we define the energy space $E_T$ as
\begin{equation*}
  E_T\eqdefa C([0, T]; X^{12}_{\frac{1}{2}}\cap H^1(\Omega))
\end{equation*}
with the instantaneous energy $\mathcal{E}(t)$ (in terms to the velocity $v$)
\begin{align*}
\mathcal{E}(t)\eqdefa\|v\|_{X^{12}_{\f12}}^2+\|v\|_{H^1}^2,
\end{align*}
and the dissipation $\mathcal{D}(t)$
\begin{align*}
\mathcal{D}(t)\eqdefa\|\na v\|_{X^{12}}^2+\|\brho^{\f12}\pa_t v\|_{L^2}^2.
\end{align*}
Given $\kappa>0$, we also introduce the space $F_{\kappa}$ in terms to the flow map $\eta$ as follows:
\begin{equation*}\label{space-flow-map-1}
\mathcal{F}_{\kappa}=\mathcal{F}_{\kappa}(\Omega)\eqdefa\{\xi\in\,X^{12}\cap\,H^1(\Omega)|\,\,\nabla\xi\in \,X^{12},\quad \overline{\rho}^{-\frac{1}{2}+\kappa}\Delta\xi\in\,L^2\}
\end{equation*}
equipped with the norm
\begin{equation*}
\|\xi\|_{\mathcal{F}_{\kappa}}\eqdefa\|\xi\|_{X^{12}}
+\|\nabla\xi\|_{X^{12}}+\|\overline{\rho}^{-\frac{1}{2}+\kappa}\Delta\xi\|_{L^2}.
\end{equation*}

Now, we are in the position to state our main results:
\begin{theorem}\label{thm: main}
Under the assumptions \eqref{assum: brho1}-\eqref{assum: brho2}, assume that  there exists a positive number $\sigma_0$ such that
\begin{align}
&dist(\Gamma(0), \Gamma_b)>0,\label{assum: intial 1}\\
&2\sigma_0\leq J_0\leq 3\sigma_0\label{assum: intial 2}.
\end{align}
If the initial data $(v_0,\, \eta_0)\in (X^{12}_{\f12}\cap H^1(\Omega)) \times \mathcal{F}_{\kappa}(\Omega)$ for some $\kappa \in (0, \frac{1}{16})$, then the system \eqref{eq:CNS1} is locally well-posed. More precisely, there exists a positive time $T>0$ such that the system \eqref{eq:CNS1} has a unique solution $(v,\,\eta)\in\,C([0, T]; X^{12}_{\f12}\cap H^1(\Omega)) \times C([0, T]; \mathcal{F}_{\kappa}(\Omega))$ depending continuously on initial data $(v_0,\, \eta_0)\in (X^{12}_{\f12}\cap H^1(\Omega)) \times \mathcal{F}_{\kappa}(\Omega)$, and there hold
\begin{equation}\label{est-v-flow-map}
\begin{split}
 &\sup_{t\in[0,T]}(\|v\|^2_{X^{12}_{\f12}}+\|v\|_{H^1}^2) +\int_0^T (\|\na v\|_{X^{12}}^2+\|\brho^{\f12}\pa_t v\|_{L^2}^2)\, ds\leq C,\\
 &\sup_{t\in[0,T]}\|\xi(t)\|^2_{\mathcal{F}_{\kappa}}\leq C,\quad\,\sigma_0\leq \sup_{(t, x)\in[0,T]\times \Omega}J(t, x)\leq 4\sigma_0,
\end{split}
\end{equation}
where $C$ depends on initial data.
\end{theorem}
\begin{remark}
The assumption \eqref{assum: brho1}-\eqref{assum: brho2} on $\rho_0$ is reasonable. If  $\Om(0)=\Om:=\{x \in \mathbb{T}^2\times \mathbb{R}|\quad 0 <x_3<1\}$ and $\rho_0=dist(x,\pa\Om)\sim x_3(1-x_3)$ then the assumptions \eqref{assum: brho1}- \eqref{assum: brho2} are automatically satisfied .
\end{remark}

\begin{remark}
In this paper, we consider the case that $\gamma=2$. But our method may still work for all the cases $\gamma>1$.
\end{remark}
\begin{remark}\label{rmk-euler-1}
For any $t \in [0, T]$, since $\sigma_0\leq \sup_{(t, x)\in[0,T]\times \Omega}J(t, x)\leq 4\sigma_0$, the flow-map $\eta(t, x)$ defines a diffeomorphism from the equilibrium domain $\Omega$ to the moving domain $\Omega(t)$ with the boundary $\Gamma(t)$. From this, together with the fact that $\eta_0$ is a diffeomorphism from the equilibrium domain $\Omega$ to the initial domain $\Omega(0)$, we deduce a diffeomorphism from the initial domain $\Omega(0)$ to the evolving domain $\Omega(t)$ for any $t \in [0, T]$. Denote the inverse of the flow map $\eta(t, x)$ by $\eta^{-1}(t, y)$ for $t \in [0, T]$ so that if $y = \eta(t, x)$ for $y \in \Omega(t)$ and $t \in [0, T]$, then $x = \eta^{-1}(t, y) \in \Omega$.
\end{remark}
For the strong solution $(\eta, v)$ obtained in Theorem \ref{thm: main}, and for $y \in \Omega(t)$ and $t \in [0, T]$, we denote that
\begin{equation}\label{euler-qu-1}
\begin{split}
&\rho(t,\,y):=J^{-1}(t, \eta^{-1}(t, y))\brho_0(\eta^{-1}(t, y)),\quad\,u(t,\,y):=v(t,\,\eta^{-1}(t, y)).
\end{split}
\end{equation}
Then the triple $(\rho(t,\,y), u(t, y), \Omega(t))$ ($t \in [0, T]$) defines a strong solution to the free boundary problem \eqref{eq:CNS-Free}. Furthermore, we obtain the following theorem.

\begin{theorem}\label{thm-euler-CNS}
Under the assumptions in Theorem \ref{thm: main}, the free boundary problem \eqref{eq:CNS-Free} is locally well-posed, and the triple  $(\rho(t,\,y), u(t, y), \Omega(t))$ ($t \in [0, T]$) defined in Remark \ref{rmk-euler-1} and \eqref{euler-qu-1} is the unique strong solution to the free boundary problem \eqref{eq:CNS-Free} satisfying $\eta-Id \in\,C([0, T],\, \mathcal{F}_{\kappa})$.
\end{theorem}

The rest of the paper is organized as follows. In Section 2, we derive some preliminary estimates. Some necessary {\it a priori} estimates are obtained in Section 3. Finally in Section 4, the proof of Theorem \ref{thm: main} is completed.

Let us complete this section with some notations that we use in this context.

\medbreak \noindent{\bf Notations:} Let $A, B$ be two operators, we denote $[A, B]=AB-BA,$ the commutator between $A$ and $B$. For $a\lesssim b$, we mean that there is a uniform constant $C,$ which
may be different on different lines, such that $a\leq Cb$ and $C_0$ denotes a positive constant depending on the initial data only.

\renewcommand{\theequation}{\thesection.\arabic{equation}}
\setcounter{equation}{0}


\section{Preliminary estimates}

In what follows, we denote by $C$ a positive constant which may depend on initial data $(v_0,\eta_0)$ and $\d$. Besides, we denote by $C_0$ and $c_0$ two generic positive constants with may depend on initial data $(v_0,\eta_0)$ and parameters of the problem, but independent of $\delta$ in the norm $\|\cdot\|_{X^{N}_{\alpha}}$. Those notation are allowed to change from one inequality to the next.
\medskip

We first introduce the following inequality which we heavily use in our paper.
\begin{lemma}(Hardy inequality, \cite{KMP2007})\label{lem: hardy}
For any $\e>0,$ there holds that
\begin{align*}
\|\brho^{-\f12+\e}f\|_{L^2}\leq C_0(\|\brho^{\f12+\e}f\|_{L^2}+\|\brho^{\f12+\e}\na f\|_{L^2}).
\end{align*}

\end{lemma}

\medskip

 With Hardy inequality in hand, we have the following interpolation equalities.

\begin{lemma}\label{lem: interpolation}
For any $\kappa\in(0,\f1{16})$, there hold that
 for $0\leq \ell\leq 6$
 \begin{align}\label{est: interpolation 1}
 \d^{|\ell|}\|Z^\ell \na f\|_{L^\infty_{x_3}(L^2_h)}\leq C(\|\na f\|_{X^{12}}+\|\brho^{-\f12+\kappa}\tri f\|_{L^2}).
 \end{align}
and for $0\leq \ell\leq 4$
\begin{align}\label{est: interpolation 2}
\d^{|\ell|} \|Z^\ell \na f\|_{L^\infty}\leq C_0(\|\na f\|_{X^{12}}+\|\brho^{-\f12+\kappa}\tri f\|_{L^2}).
 \end{align}

\end{lemma}

\begin{proof}
For  $0\leq \ell\leq 6$, thanks to Sobolev embedding theorem and Lemma \ref{lem: hardy}, we have
\begin{align*}
 &\d^{|\ell|}\|Z^\ell \na f\|_{L^\infty_{x_3}(L^2_h)}
 \leq  C_0 \d^{|\ell|}(\|\brho^{-\f{21}{44}}Z^\ell \na f\|_{L^2_{x_3}(L^2_h)}+ \|\brho^{\f{21}{44}}\pa_3 Z^\ell \na f\|_{L^2_{x_3}(L^2_h)}) \\
 &\leq C_0 \d^{|\ell|}(\|\brho^{\f{23}{44}}Z^\ell \na f\|_{L^2}+\|\brho^{\f{23}{44}}\na Z^\ell \na f\|_{L^2}+ \sum_{i=0}^\ell(\|\brho^{\f{21}{44}}Z^{i+1} \na f\|_{L^2}+ \|\brho^{\f{21}{44}}Z^i\tri f\|_{L^2})\\
 &\leq C_0\|\na f\|_{X^{12}}+C_0 \d^{|\ell|} \sum_{i=0}^\ell  \|\brho^{\f{21}{44}}Z^i\tri f\|_{L^2}.
\end{align*}
Thanks to $|Z \bar\rho|\leq C\bar\rho$, we deduce from integration by parts that
\beno
 \delta^i\|\brho^{\f{21}{44}}Z^i\tri f\|_{L^2}&\leq& C_0\|\brho^{-\f12+\kappa}\tri f\|_{L^2}^{1-\f{7}{11}} \|\tri f\|_{\dot{X}^{11}_1}^{\f{7}{11}}\\
 &\leq& C_0(\|\brho^{-\f12+\kappa}\tri f\|_{L^2}+ \|\tri f\|_{\dot{X}^{11}_1}),\quad \forall\,\,\, i\leq 5,
\eeno
where we used that $$\f{14}{22}+(-\f12+\kappa)\f{4}{11}\leq \f{21}{44}<\f12.$$

While by using integration by parts again, one can see
\beno
 \delta^{12}\|\brho^{\f{21}{44}}Z^6\tri f\|^2_{L^2}&=&\delta^{12}\int_{\Om}\bar\rho^{\f{21}{22}} Z^6\tri f\cdot Z^6\tri f\, dx\\
 &\leq& C_0 \delta\|\tri f\|_{X^{11}_1}(\|\bar\rho^{-\f{1}{22}}\tri f\|_{L^2}+\|\bar\rho^{-\f{1}{22}}Z\tri f\|_{L^2}).
\eeno
Next, we deal with the last term in the above inequality. In fact, we have
 \begin{equation*}
   \begin{split}
 &\|\bar\rho^{-\f{1}{22}}Z\tri f\|_{L^2}^2=\int_{\Om}\bar\rho^{-\f{1}{11}}Z\tri f\cdot Z\tri f\,dx \\
 &\leq  C_0\|\rho^{-\f12+\kappa}\tri f\|_{L^2} \,\sum_{k=0}^2\|\rho^{\f12-\f{1}{11}-\kappa} Z^k\tri f \|_{L^2}\\
 &\leq  C_0\|\rho^{-\f12+\kappa}\tri f\|_{L^2} (\|\tri f\|_{X^{11}_1}+C_0\|\brho^{-\f12+\kappa}\tri f\|_{L^2}),
   \end{split}
 \end{equation*}
which implies
 \begin{equation*}
   \begin{split}
 &\|\bar\rho^{-\f{1}{22}}Z\tri f\|_{L^2}\leq  C_0(\|\tri f\|_{X^{11}_1}+C_0\|\brho^{-\f12+\kappa}\tri f\|_{L^2}).
   \end{split}
 \end{equation*}
Combining all the above estimates, we get that for $0\leq \ell\leq 6$
\begin{align*}
 \d^{|\ell|}\|Z^\ell \na f\|_{L^\infty_{x_3}(L^2_h)}
   \leq C_0(\|\na f\|_{X^{12}}+\|\brho^{-\f12+\kappa}\tri f\|_{L^2}),
\end{align*}
that is, the inequality \eqref{est: interpolation 1} holds.

\medskip

The second inequality \eqref{est: interpolation 2} comes from Sobolev embedding theorem and \eqref{est: interpolation 1}
\begin{align*}
\d^{|\ell|}\|Z^\ell \na f\|_{L^\infty_x}\leq C_0\d^{|\ell|} \|Z^\ell \na f\|_{L^\infty_{x_3}H^2_h}\leq  C_0(\|\na f\|_{X^{12}}+\|\brho^{-\f12+\kappa}\tri f\|_{L^2}).
\end{align*}
 for $0\leq \ell\leq 4$, which ends the proof of Lemma \ref{lem: interpolation}.
\end{proof}

\medskip

To deal with nonlinear term, we need the following product estimates:
\begin{lemma}\label{lem: product2}
It holds that
\begin{align}\label{lem: pro1}
\|g~f\|_{X^{12}}\leq&  C_0\|g\|_{X^{12}} \sum_{|\ell|\leq 6}\d^{|\ell|} \|Z^\ell f\|_{L^\infty_{x_3}(L^2_h)}+C_0 \| f\|_{X^{12}}\sum_{|\ell|\leq 6}\d^{|\ell|} \|Z^\ell g\|_{L^\infty_{x_3}(L^2_h)}.
\end{align}
and
\begin{align}\label{lem: pro2}
\sum_{0\leq|j|\leq 1}\|g~Z^{j}f\|_{X^{11}}\leq& C_0\|g\|_{X^{11}} \sum_{|\ell|\leq 6}\d^{|\ell|} \|Z^\ell f\|_{L^\infty_{x_3}(L^2_h)}+C_0\| f\|_{X^{12}}\sum_{|\ell|\leq 6}\d^{|\ell|}  \|Z^\ell g\|_{L^\infty_{x_3}(L^2_h)},
\end{align}
where the constants $C$ may depend on $\delta$.
\end{lemma}
\begin{proof}
By the Leibnitz formula, one can see
\begin{align*}
\|g~f\|_{X^{12}}=\Big(\sum_{|m|=0}^{12}\d^{2|m|}\|Z^m(g~f)\|_{L^2}^2\Big)^{\f12}\leq C_0\sum_{|m_1|+|m_2|=0}^{12}\d^{|m_1|+|m_2|}\|Z^{m_1}g~Z^{m_2}f\|_{L^2}.
\end{align*}

Now, we focus on the most difficulty case: $|m_1|+|m_2|=12$. The others can be treated by a similar way. We divide its proof into three cases.

{\bf Case 1.} $8\leq |m_1|\leq 12$.  By H\"{o}lder's inequality, we get
\beno
\d^{12}\|Z^{m_1}g~ Z^{m_2}f\|_{L^2}& \leq& \d^{12} \|Z^{m_1}g \|_{L^2}\|Z^{m_2}f\|_{L^\infty}
\leq C_0\d^{12} \|Z^{m_1}g \|_{L^2}\|Z^{m_2}f \|_{L^\infty_{x_3}(H^2_h)}\\
&\leq& C_0\|g\|_{X^{12}} \sum_{|\ell|\leq 6}\|Z^\ell f\|_{L^\infty_{x_3}(L^2_h)}\d^{|\ell|},
\eeno
where we used $|m_2|+2\leq 6$.

{\bf Case 2.}   $6\leq |m_1|\leq 7$. Thanks to Sobolev embedding theorem and H\"{o}lder's inequality, we get
\beno
\d^{12}\|Z^{m_1}g~ Z^{m_2}f\|_{L^2}& \leq& \d^{12} \|Z^{m_1}g \|_{L^2_{x_3}(L^\infty_{h})}\|Z^{m_2}f\|_{L^\infty_{x_3}(L^2_h)}
\\
&\leq& C_0\|g\|_{X^{12}} \sum_{|\ell|\leq 6}\|Z^\ell f\|_{L^\infty_{x_3}(L^2_h)}\d^{|\ell|}.
\eeno

{\bf Case 3.}   $ 0\leq |m_1|\leq 5$.
 For this case, we only need to change the position of $f$ and $g$ and apply the same argument as the above two cases to get that
\beno
\d^{12}\|Z^{m_1}g~ Z^{m_2}f\|_{L^2}&\leq& C_0\|f\|_{X^{12}} \sum_{|\ell|\leq 6}\d^{|\ell|}\|Z^\ell g\|_{L^\infty_{x_3}(L^2_h)}.
\eeno

\medskip

Collecting all cases together, we obtain
\begin{align*}
\|g~f\|_{X^{12}}\leq&  C_0\|g\|_{X^{12}} \sum_{|\ell|\leq 6}\d^{|\ell|}\|Z^\ell f\|_{L^\infty_{x_3}(L^2_h)}+C_0\sum_{|\ell|\leq 6}\d^{|\ell|} \|Z^\ell g\|_{L^\infty_{x_3}(L^2_h)}\| f\|_{X^{12}},
\end{align*}
which follows \eqref{lem: pro1}.

Next, since we the highest order in \eqref{lem: pro2} is 11, we may readily verify \eqref{lem: pro2} by the same process above, which ends the proof of Lemma \ref{lem: product2}.
\end{proof}

\medskip

We introduce a new quantity $\frak{ D}(v)(t)$ which controls $\|\na v\|_{L^\infty}$ from Lemma \ref{lem: interpolation}:
\begin{align}\label{def: D(v)(t)}
\frak{ D}(v)(t)\eqdefa \|\na v(t)\|_{X^{12}}+\|\brho^{-\f12+\kappa}\tri v(t)\|_{L^2}.
\end{align}

In what follows, $\mathcal{P}(\cdot)$ stands for some polynomial function which coefficients may depend on $\d$.

\medskip

\begin{lemma}\label{lem: cA,JcA}
Assume that
\begin{align*}
&\xi_0\in \mathcal{F}_{\kappa}, \quad \|\frak{ D}(v)\|_{L^2(0,  T)}\leq \frak{C},\quad \sigma_0\leq J\leq 4\sigma_0.
\end{align*}
Then there hold that for any $t \in [0, T]$
\begin{equation}\label{est-JA-1}
  \begin{split}
&\sum_{0\leq |\ell|\leq 6}\d^{|\ell|}\|Z^\ell(J\cA)(t)\|_{L^\infty_{x_3}(L^2_h)}\leq C_0(1+t^{\f12}\mathcal{P}(\frak{C})),\\
&\sum_{0\leq |\ell|\leq 6}\d^{|\ell|} \|Z^\ell \cA(t)\|_{L^\infty_{x_3}(L^2_h)} \leq C_0(1+t^{\f12}\mathcal{P}(\frak{C})),\quad\|\na v:\na v(t)\|_{X^{12}}\leq C_0\frak{ D}(v)^2(t),\\
&\|J\cA (t)\|_{X^{12}}\leq  C_0(1+t^{\f12}\mathcal{P}(\frak{C})),\quad \|\cA(t) \|_{X^{12}}\leq C_0(1+t^{\f12}\mathcal{P}(\frak{C})),
  \end{split}
\end{equation}
where the constant $C_0$ depends on $\|\xi_0\|_{\mathcal{F}_{\kappa}}$ and $\sigma_0$.
\end{lemma}
\begin{proof}
Before giving the proof of this lemma, we state some estimates as preliminary. By Lemma \ref{lem: interpolation}, one can prove that
\begin{equation}\label{est: (na v)^2 low}
  \begin{split}
&\sum_{0\leq |\ell|\leq 6}\d^{|\ell|}\|Z^\ell(\na v:\na v)\|_{L^\infty_{x_3}(L^2_h)}\leq C_0\sum_{0\leq |\ell|\leq 6}\|Z^\ell\na v\|_{L^\infty_{x_3}(L^2_h)} \sum_{0\leq |\ell|\leq 4}\d^{|\ell|}\|Z^\ell \na v\|_{L^\infty}\\
&\leq\,C_0(\|\na v\|_{X^{12}}+\|\brho^{-\f12+\kappa}\tri v\|_{L^2})^2\leq\,C_0\frak{ D}(v)^2.
  \end{split}
\end{equation}

Taking $f=g=\na v$ in \eqref{lem: pro1} and using  Lemma \ref{lem: interpolation}, we obtain
\begin{align}\label{est: (na v)^2 high}
\|\na v:\na v\|_{X^{12}}\leq C\|\na v\|_{X^{12}} \sum_{|\ell|\leq 6}\d^{|\ell|}\|Z^\ell \na v\|_{L^\infty_{x_3}(L^2_h)}\leq C\frak{ D}(v)^2.
\end{align}

Now we are in the position to prove the first estimates. Notice that
\begin{align*}
J\cA=(D\eta)^{-1}=(\na \eta_0+\int_{0}^{t}\na v\,ds)^{-1},
\end{align*}
and every entry in $J\cA$ is a linear combination of
$$\na \eta_0,\, \na \eta_0\int_{0}^{t}\na vds,\,(\int_{0}^{t}\na vds)^2.$$
Then, thanks to Lemma \ref{lem: interpolation}-\ref{lem: product2}, \eqref{est: (na v)^2 low} and Minkowski's inequality, one has
\begin{equation}\label{est:JcA-low-1}
  \begin{split}
&\sum_{0\leq |\ell|\leq 6}\d^{|\ell|}\|Z^\ell(J\cA)\|_{L^\infty_{x_3}(L^2_h)}\\
&\leq \sum_{0\leq |\ell|\leq 6}\d^{|\ell|}\|Z^\ell\na \eta_0\|_{L^\infty_{x_3}(L^2_h)}+\sum_{0\leq |\ell|\leq 6}\d^{|\ell|}\|Z^\ell(\na \eta_0\int_{0}^{t}\na vds)\|_{L^\infty_{x_3}(L^2_h)}\\
&\qquad\qquad+\sum_{0\leq |\ell|\leq 6}\d^{|\ell|}\|Z^\ell((\int_{0}^{t}\na vds)^2)\|_{L^\infty_{x_3}(L^2_h)}\\
&\leq C_0\,\|\xi_0\|_{\mathcal{F}_{\kappa}}+C_0\,\|\xi_0\|_{\mathcal{F}_{\kappa}}\,t^{\f12}\|\frak{ D}(v)\|_{L^2_t}+C_0t\|\frak{ D}(v)\|_{L^2_t}^2
\leq  C_0(1+t\,\mathcal{P}(\frak{C})),
  \end{split}
\end{equation}
which proves the first inequality in \eqref{est-JA-1}.

Similarly, we deduce
\begin{align}\label{est: JcA low t}
&\sum_{0\leq |\ell|\leq 6}\d^{|\ell|}\|Z^\ell(\na\eta_0\int_{0}^{t}\na vds)\|_{L^\infty_{x_3}(L^2_h)}+\sum_{0\leq |\ell|\leq 6}\d^{|\ell|}\|Z^\ell((\int_{0}^{t}\na vds)^2)\|_{L^\infty_{x_3}(L^2_h)} \\
\nonumber
&\qquad\leq C_0t^{\f12}\|\frak{ D}(v)\|_{L^2_t}+C_0t\|\frak{ D}(v)\|_{L^2_t}^2\leq C_0t^{\f12} \mathcal{P}(\frak{C}).\nonumber
\end{align}

Recalling the definition of $J$:
\begin{equation*}\label{eq: J}
J=det(\na \eta_0+\int_0^t\na vds),
\end{equation*}
$J$ is a linear combination of the terms
$$(\na \eta_0)^3,\,\na \eta_0(\int_{0}^{t}\na vds)^2,\, (\na \eta_0)^2\int_{0}^{t}\na vds,\,(\int_{0}^{t}\na vds)^3.$$
Hence, similar to the proof of the first inequality in \eqref{est-JA-1} in terms of $J\cA$, we may obtain
\begin{align}\label{est : J low}
\sum_{0\leq |\ell|\leq 6}\d^{|\ell|}\|Z^\ell J\|_{L^\infty_{x_3}(L^2_h)}\leq   C_0(1+t^{\f12}\mathcal{P}(\frak{C})).
\end{align}
Owing to $J\geq \sigma_0$ and  the formula to the composition of two functions, we obtain
\begin{align*}
\sum_{0\leq |\ell|\leq 6}\d^{|\ell|}\|Z^\ell (J^{-1})\|_{L^\infty_{x_3}(L^2_h)}\leq&C_0\sum_{0\leq |\ell|\leq 6}\d^{|\ell|}\|\prod_{\sum_j |k_j| m_j\leq |\ell|}(Z^{k_j}J)^{m_j}\|_{L^\infty_{x_3}(L^2_h)}
\end{align*}
We put $\|\cdot \|_{L^\infty_{x_3}(L^2_h)}$ on the highest order term $Z^{k_j}J$ and put $\|\cdot\|_{L^\infty}$ to other lower terms (not more than order 4) with similar process to \eqref{est: (na v)^2 low}. It follows from Lemma \ref{lem: interpolation} and \eqref{est : J low} that
\begin{align}\label{est: J^-1 low}
\sum_{0\leq |\ell|\leq 6}\d^{|\ell|}\|Z^\ell (J^{-1})\|_{L^\infty_{x_3}(L^2_h)}\leq C_0(1+t^{\f12}\mathcal{P}(\frak{C})).
\end{align}
Therefore, due to \eqref{est:JcA-low-1} and \eqref{est: J^-1 low}, we infer
\begin{equation*}\label{est: cA low}
\begin{split}
\sum_{|\ell|\leq 6}\d^{|\ell|}\|Z^{\ell}\cA\|_{L^\infty_{x_3}(L^2_h)}\leq &C_0+C\sum_{|\ell|\leq 6}\d^{|\ell|}\|Z^{\ell}(J\cA)\|_{L^\infty_{x_3}(L^2_h)}\sum_{|\ell|\leq 4}\d^{|\ell|}\|Z^{\ell}(J^{-1})\|_{L^\infty}\\
&+C\sum_{|\ell|\leq 6}\|Z^{\ell}(J^{-1})\|_{L^\infty_{x_3}(L^2_h)}\d^{|\ell|}\sum_{|\ell|\leq 4}\|Z^{\ell}(J\cA)\|_{L^\infty}\d^{|\ell|}\\
\leq &   C_0(1+t^{\f12}\mathcal{P}(\frak{C})).
\end{split}
\end{equation*}

 For the high order estimate, similar to the proof of \eqref{est: (na v)^2 high}, by using Lemma \ref{lem: interpolation}, we deduce
 \begin{align}\label{est: JcA high t}
\|\na \eta_0(\int_{0}^{t}\na vds)^2\|_{X^{12}}+\|(\na \eta_0)^2\int_{0}^{t}\na vds\|_{X^{12}}+\|(\int_{0}^{t}\na vds)^3\|_{X^{12}}\leq C_0\,t^{\f12}\mathcal{P}(\frak{C}),
\end{align}
and then
\begin{equation}\label{est: J,cA high}
  \begin{split}
&\|(J,\,J\cA)\|_{X^{12}}\leq C_0\Big(\|\na \eta_0\|_{X^{12}}+\|\na \eta_0(\int_{0}^{t}\na vds)^2\|_{X^{12}}\\
&\qquad\qquad\qquad\qquad\qquad\qquad+\|(\na \eta_0)^2\int_{0}^{t}\na vds\|_{X^{12}}+\|(\int_{0}^{t}\na vds)^3\|_{X^{12}}\Big)\\
&\leq C_0+C_0t^{\f12}\|\frak{ D}(v)\|_{L^2_t}+C_0t\|\frak{ D}(v)\|_{L^2_t}^2+C_0t^\f32\|\frak{ D}(v)\|_{L^2_t}^3\\
&\leq    C_0(1+t^{\f12}\mathcal{P}(\frak{C})).
  \end{split}
\end{equation}

While thanks to \eqref{est : J low}, \eqref{est: J,cA high} and Lemma \ref{lem: product2},  it follows that
\begin{align*}
\|J^{-1}\|_{X^{12}}\leq   C_0(1+t^{\f12}\mathcal{P}(\frak{C})),
\end{align*}
and
\begin{align*}\label{est: cA high1}
\|\cA\|_{X^{12}}=\|J\cA~J^{-1}\|_{X^{12}} \leq    C_0(1+t^{\f12}\mathcal{P}(\frak{C})),
\end{align*}
which completes the proof of Lemma \ref{lem: cA,JcA}.

\end{proof}

\medskip

Based on the above lemma, we may get the following estimates:
\begin{lemma}\label{lem:high nav cA}
Under the assumptions in Lemma \ref{lem: cA,JcA}, there hold
\begin{equation}\label{est-high-order}
\begin{split}
&\sum_{0\leq|j|\leq 1}\|Z^j(J\cA)~\na v \|_{X^{11}}\leq  C_0\|\na v\|_{X^{11}}+t^\f12\mathcal{P}(\frak{C})\frak{ D}(v),\\
&\sum_{0\leq|j|\leq 1}\|Z^j\cA~\na v \|_{X^{11}}\leq C_0\|\na v\|_{X^{11}}+t^\f12\mathcal{P}(\frak{C})\frak{ D}(v),\\
&\|\cA~\na v\|_{X^{12}}\leq C_0\|\na v\|_{X^{12}}+t^\f12\mathcal{P}(\frak{C})\frak{ D}(v),\\
&\|J~\mathrm{S}_{\cA} (v)\|_{X^{12}}\leq  C_0\|\na v\|_{X^{12}}+t^\f12\mathcal{P}(\frak{C})\frak{ D}(v).
\end{split}
\end{equation}
\end{lemma}

\begin{proof}
We mainly utilize Lemmas \ref{lem: product2}, \ref{lem: cA,JcA} to prove \eqref{est-high-order}. So one may only focus on the proof of the first inequality in \eqref{est-high-order}, and the proofs of the others are the same as it, whose details will be omitted here.

First, by the definition of $J\cA$, we split $\sum_{0\leq|j|\leq 1}\|\na v ~Z^j(J\cA)\|_{X^{11}}$ into three parts:
\begin{equation}\label{est-high-order-I1-3}
\begin{split}
&\sum_{0\leq|j|\leq 1}\|Z^j(J\cA)~\na v \|_{X^{11}} \leq     C\sum_{0\leq|j|\leq 1}  \|\na v\cdot Z^j\na \eta_0 \|_{X^{11}}\\
& \qquad +C\sum_{0\leq|j|\leq 1} \|\na v\cdot Z^j\na \eta_0\int_{0}^{t}\na vds \|_{X^{11}}\big)+  C\sum_{0\leq|j|\leq 1} \|\na v\cdot Z^j(\int_{0}^{t}\na vds )^2\|_{X^{11}} \triangleq \sum_{i=1}^3 I_i.
\end{split}
\end{equation}

 For $I_1$, we have
\begin{align}\label{I1-1}
I_1\leq C_0\|\na v\|_{X^{11}}.
\end{align}

For $I_2$, taking $g=\na v$ and $f=J\cA$ in \eqref{lem: pro2} in Lemma \ref{lem: product2} to obtain that
\begin{align*}
I_2\leq &C\|\na v \|_{X^{11}}\sum_{|\ell|\leq 6}\|Z^j\na \eta_0\int_{0}^{t}\na vds \|_{L^\infty_{x_3}(L^2_h)}+C \sum_{|\ell|\leq 6}\|Z^\ell \na v\|_{L^\infty_{x_3}(L^2_h)}\|Z^j\na \eta_0\int_{0}^{t}\na vds \|_{X^{12}}.
\end{align*}
Applying Lemma \ref{lem: interpolation} and \eqref{est: JcA low t}, \eqref{est: JcA high t} in Lemma \ref{lem: cA,JcA} to get
\begin{align}\label{I2-1}
I_2\leq t^{\f12} \mathcal{P}(\frak{C})\|\na v \|_{X^{11}}+
Ct^\f12\mathcal{P}(\frak{C})\frak{ D}(v)
\leq\,t^\f12\mathcal{P}(\frak{C})\frak{ D}(v).
\end{align}
Similarly, we have
\begin{align}\label{I3-1}
I_3
\leq&t^\f12\mathcal{P}(\frak{C})\frak{ D}(v).
\end{align}
Plugging the estimates \eqref{I1-1}-\eqref{I3-1} into \eqref{est-high-order-I1-3}, we prove
\begin{align*}
\sum_{0\leq|j|\leq 1}\|Z^j(J\cA)~\na v \|_{X^{11}}\leq& C_0\|\na v\|_{X^{11}}+t^\f12\mathcal{P}(\frak{C})\frak{ D}(v),
\end{align*}
which ends our proof.
\end{proof}

Next we recall a version of Korn's inequality involving only the deviatoric part $\mathbb{D}^0$.
\begin{lemma}\label{lem-korn-2}[Korn's lemma, Theorem 1.1 in \cite{Dain2006} ]
Let $n \geq 3$ and $U$ be a Lipschitz domain in $\mathbb{R}^n$, then there exists a constant $C_0$, independent of $f$, such that
\begin{equation*}\label{korn-2}
\|f\|_{H^1(U)} \leq C_0(\|\mathbb{D}^0(f)\|_{L^2(U)}+\|f\|_{L^2(U)})
\end{equation*}
for all $f \in H^1(U)$.
\end{lemma}

\renewcommand{\theequation}{\thesection.\arabic{equation}}
\setcounter{equation}{0}


\section{A priori estimates}

In this section, we give a priori estimates of the system \eqref{eq:CNS1}. The main result of the section is as follows:
\begin{proposition}\label{pro: main}
Assume $(\xi,\,v)$ is a smooth solution of system \eqref{eq:CNS1} on $[0,\bar{T}]$ with initial data $(\xi_0,\,v_0) \in \mathcal{F}_{\kappa} \times (X^{12}_{\frac{1}{2}}\cap\,H^1)$ and $0<2\sigma_0\leq J_0\leq 3\sigma_0$, and $\brho$ satisfies \eqref{assum: brho1}-\eqref{assum: brho2}.
Then, there exists a positive constant $T\leq \bar T$ which depends on the initial data such that
 \begin{align*}
 &\sup_{t\in[0,T]}\mathcal{E}(t)+\int_0^T\mathcal{D}(s)ds\leq 2\mathcal{E}(0).
\end{align*}

\end{proposition}

\bigskip

Here, we use the bootstrap argument to prove this proposition.  Now, we define a $T$ such that there holds that
\begin{align}
&\|\frak{ D}(v)\|_{L^2(0, T)}\leq \frak{C},\quad \sigma_0\leq \sup_{t\in[0, T]} J\leq 4\sigma_0.\label{assum: energy3}
\end{align}

\medskip

Before, we give the proof of the proposition, we prove some useful lemmas.

\begin{lemma}\label{lem:assumption}
Under the assumption of Proposition \ref{pro: main}, we have
\beno
\|\na v\|_{L^1(0, t; L^\infty)} \leq t^\f12\mathcal{P}(\frak{C}),\quad \|(J, A)(t)\|_{L^\infty}   \leq C_0(1+t^{\f12} \mathcal{P}(\frak{C})) ,\quad \forall \,\, t\in [0, T).
\eeno
\end{lemma}
\begin{proof}
It is a direct result from Lemma \ref{lem: interpolation} and Lemma \ref{lem: cA,JcA}.
\end{proof}

%

\medskip

\begin{lemma}\label{Lem: H^1}
Under the assumption of Proposition \ref{pro: main}, there exists a constant $\delta_0$ which depends on the initial data such that for $\delta\leq \delta_0$, the following holds
\begin{equation}\label{korn-ine-Xn-1}
\|\na v\|_{X^{N}}\leq C_0(\|\mathbb{D}^0(v)\|_{X^{N}}+\|v\|_{X^N_{\f12}}).
\end{equation}

\end{lemma}
\begin{proof}
Thanks to Korn's lemma (Lemma \ref{lem-korn-2}), we have
\begin{equation*}
\begin{split}
&\|v\|_{H^1}\leq C_0(\|\mathbb{D}^0(v)\|_{L^2}+\|v\|_{L^2}).
\end{split}
\end{equation*}
For any function $f(s)$, by Lemma \ref{lem: hardy}, we have
\begin{align*}
\int_0^1f^2ds\leq& C_0\int_0^1s^2(f^2+f'^2)ds\\
\end{align*}
By scaling, we have
\begin{align*}
\int_{1-\e}^1 f^2ds\leq& \f{C_0}{\e^2}\int_{1-\e}^1 (1-s)^2f^2ds+C_0\int_{1-\e}^1 ({1-s})^2f'^2ds.
\end{align*}
Then \eqref{assum: brho3} gives that
\begin{align}\label{est: v L^2}
\|f\|_{L^2}^2\leq C_0\|{\brho}^{-1}\|_{L^\infty(0\leq x_3\leq 1-\e)}\int _{\Om}\brho f^2dx+C_0\e^2\|f\|_{H^1}^2\leq \f{C_0}{\e}\int _{\Om}\brho f^2dx+C_0\e^2\|f\|_{H^1}^2.
\end{align}
Taking $\e$ small enough and $f:=v$, we  combine with Lemma \ref{lem-korn-2} to get that
\begin{align*}
\|v\|_{H^1}\leq C_0(\|\mathbb{D}^0(v)\|_{L^2}+\|\brho^\f12v\|_{L^2}).
\end{align*}

For given $m\in\mathbb{N}^3$: $1\leq|m| \leq N$,
\begin{equation*}
\begin{split}
\d^{|m|}\|Z^mv\|_{H^1}\leq& C_0\d^{|m|}(\|\mathbb{D}^0(Z^mv)\|_{L^2}+\|\brho^\f12Z^mv\|_{L^2})\\
\leq& C_0\d^{|m|}(\|Z^m\mathbb{D}^0v\|_{L^2} +\|[\mathbb{D}^0,Z^m]v\|_{L^2}+\|\brho^\f12 Z^mv\|_{L^2}),
\end{split}
\end{equation*}
which follows from the fact $[\mathbb{D}^0,Z^m]v\sim Z^{m-1}\, \nabla\,v$ that
\begin{equation*}
\begin{split}
&\d^{|m|}\|Z^mv\|_{H^1}\leq C\d^{|m|}(\|Z^m\mathbb{D}^0v\|_{L^2} +\|Z^{m-1}\nabla\,v\|_{L^2}+\|\brho^\f12 Z^mv\|_{L^2}).
\end{split}
\end{equation*}
Therefore, summing $|m|$ from 0 to $N$ and the definition of space $X^{N}$, we take $\d$ so small to arrive at \eqref{korn-ine-Xn-1}.
\end{proof}
\medskip

\begin{lemma}\label{lem-equiv-coor-1}
Let the initial flow map $\eta_0=Id+\xi_0: \Omega\rightarrow \Omega(0)$ satisfy its Jacobian $2\sigma_0\leq\,J_0\leq 3\sigma_0$ and $\xi_0 \in \mathcal{F}_{\kappa}$, and its inverse map $\eta_0^{-1}: \Omega(0)\rightarrow \Omega$, $v(x)=\widetilde{u}(\eta_0(x))$ with $x \in \Omega$ and  $\widetilde{u}(y)=v(\eta_0^{-1}(y))$ with $y \in \Omega(0)$, then there is a positive constant $C_1\geq 1$ such that
\begin{equation}\label{est-equiv-H1-1}
  \begin{split}
  &C_1^{-1} (1+\|\xi_0\|_{\mathcal{F}_{\kappa}}^2)^{-1}\int_{\Omega}|\nabla\,v|^2\,dx \leq\int_{\Omega(0)}|\nabla_y\,\widetilde{u}(y)|^2\,dy\leq\,C_1 (1+\|\xi_0\|_{\mathcal{F}_{\kappa}}^2)\int_{\Omega}|\nabla\,v|^2\,dx.
  \end{split}
\end{equation}
\end{lemma}
\begin{proof}
First, taking changes of variables $y=\eta_0(x)$, we have
\begin{equation*}
  \begin{split}
  &\int_{\Omega(0)}|\nabla_y\,\widetilde{u}(y)|^2\,dy=\int_{\Omega}|\nabla_y\,v(x)|^2\,d(\eta_0(x))
  =\int_{\Omega(0)}|(D_y(\eta_0^{-1}))(\eta_0(x))\nabla_x\,v(x)|^2\,J_0dx,
  \end{split}
\end{equation*}
which along with the assumptions $2\sigma_0\leq\,J_0\leq 3\sigma_0$, $\xi_0 \in \mathcal{F}_{\kappa}$, and \eqref{est: interpolation 2} implies
\begin{equation*}\label{est-equiv-H1-1a}
  \begin{split}
  &\int_{\Omega(0)}|\nabla_y\,\widetilde{u}(y)|^2\,dy \leq C \|(D_y(\eta_0^{-1}))(\eta_0(x))\|_{L^\infty}^2\int_{\Omega}|\nabla_x\,v(x)|^2\,dx\\
  &\leq C_1(1+\|\xi_0\|_{\mathcal{F}_{\kappa}}^2)\int_{\Omega}|\nabla\,v|^2\,dx.
  \end{split}
\end{equation*}
Similarly, one may readily check
\begin{equation*}\label{est-equiv-H1-1b}
  \begin{split}
  &\int_{\Omega}|\nabla_x\,v(x)|^2\,dx
  =\int_{\Omega(0)}|(D_x\eta_0)(\eta_0^{-1}(y))\nabla_y\,\widetilde{u}(y)|^2\,J_0^{-1}dy\\
  &\leq\,C_1(1+\|\xi_0\|_{\mathcal{F}_{\kappa}}^2)
  \int_{\Omega(0)}|\nabla_y\,\widetilde{u}(y)|^2\,dy.
  \end{split}
\end{equation*}
Therefore, we get \eqref{est-equiv-H1-1}, and complete the proof of Lemma \ref{lem-equiv-coor-1}.
\end{proof}

\begin{lemma}\label{lem: equal na v}
Under the assumption of Proposition \ref{pro: main}, if \eqref{assum: energy3} holds, then we have
\begin{align*}
(c_0- t^{\frac{1}{2}} \mathcal{P}(\frak{C}))\|\na v\|_{L^2}^2-C_0\|\brho^\f12 v\|_{L^2}^2\leq\|\mathbb{D}_{\cA}^0 v\|_{L^2}^2\leq C_0(1+t^{\frac{1}{2}} \mathcal{P}(\frak{C}))\|\na v\|_{L^2}^2.
\end{align*}

Moreover, if $T$ small enough such that $T^{\frac{1}{2}}\mathcal{P}(\frak{C})<\f{c_0}2$, then we have
\begin{equation*}
\begin{split}
\int_\Om J\mathbb{S}_\cA v: \na_{\cA}v \,dx&\geq c_1\|\mathbb{D}_\cA^0 v\|_{L^2}^2 \geq \f{c_0c_1}{2}\|\na v\|_{L^2}^2-C_0\|\brho^\f12 v\|_{L^2}^2.
\end{split}
\end{equation*}

\end{lemma}

\begin{proof}
We first to prove the first result. According to the fact
\begin{align}\label{eq: JcA-I}
J\cA-J_0\cA_0\sim (\int_0^t \na vds)^2,
\end{align}
and $\cA_0^{-1}=D\eta_0,$
combining Lemmas \ref{lem: interpolation}, \ref{lem:assumption} with \eqref{assum: energy3} , we have
\begin{align}\label{est: JcA-I}
\|J\cA-J_0\cA_0\|_{L^\infty}\leq& C_0\|\na v\|_{L_t^1L^\infty}^2\leq C_0 t^{\frac{1}{2}} \mathcal{P}(\frak{C}),\quad \|(\cA_0^{-1},\cA_0)\|_{L^\infty}\leq C_0(1+t^{\f12}\mathcal{P}(\frak{C})),
\end{align}
which imply that
\begin{align}\label{est-J-J0-1}
&\|\mathbb{D}^0_{J\cA-J_0\cA_0}(v)\|_{L^2}^2\leq C_0\|J\cA-J_0\cA_0\|_{L^\infty}^2\|\na v\|_{L^2}^2\leq t^{\frac{1}{2}} \mathcal{P}(\frak{C})\|\na v\|_{L^2}^2,\\
&\|\mathbb{D}_{J_0\cA_0}^0(v)\|_{L^2}^2\leq C_0\|\na  v\|_{L^2}^2.
\end{align}
On the other hand, we use \eqref{assum: energy3}, the coordinate transformation from $\Om$ to $\Om(0)$ and Lemmas \ref{lem-korn-2}, \ref{lem-equiv-coor-1} to get that
\begin{equation*}
  \begin{split}
  &\int_{\Omega}|\mathbb{D}^0_{\cA_0}(v)|^2\,J_0\,dx=\int_{\Omega(0)}|\mathbb{D}^0(\widetilde{u})|^2\,dx\geq c_1\|\nabla\,\widetilde{u}\|_{L^2(\Om(0))}^2-C_1\|\widetilde{u}\|_{L^2(\Om(0))}^2\\
  &\geq c_0\|v\|_{H^1}^2-C_0\|v\|_{L^2}^2,
  \end{split}
\end{equation*}
where  $\widetilde{u}=v\circ\eta_0^{-1}$. Hence, according to \eqref{assum: energy3} and \eqref{est: v L^2}, we obtain that
\begin{align*}
\|\mathbb{D}^0_{J_0\cA_0} (v)\|_{L^2}^2\geq c_0\|v\|_{H^1}^2-C_0\|\brho^\f12 v\|_{L^2}^2,
\end{align*}
which combining with \eqref{est-J-J0-1} gives rise to
\begin{align*}
(c_0- t^{\frac{1}{2}} \mathcal{P}(\frak{C}))\|\na v\|_{L^2}^2-C_0\|\brho^\f12 v\|_{L^2}^2\leq\|\mathbb{D}_{\cA}^0 v\|_{L^2}^2\leq (C_0+t^{\frac{1}{2}} \mathcal{P}(\frak{C}))\|\na v\|_{L^2}^2,
\end{align*}
which we complete the first result. For the second one, we deduce
\begin{equation*}
\begin{split}
\int_\Om J\mathbb{S}_\cA v: \na_{\cA}v dx&=\f12\int_{\Om}(\f\mu2|\mathbb{D}^0_{\cA} v|^2+(\lambda+\frac{2}{3}\mu) |\na_\cA\cdot v|^2)\,Jdx\\
  &\geq c_1\|\mathbb{D}^0_{\cA} v\|_{L^2}^2\geq (c_0c_1- t^{\frac{1}{2}} \mathcal{P}(\frak{C}))\|\na v\|_{L^2}^2-C_0\|\brho^\f12 v\|_{L^2}^2,
\end{split}
\end{equation*}
here we use \eqref{assum: energy3} in the  last step and assumption  $\mu>0,~\lambda+\frac{2}{3}\mu\geq 0$. Combining with the first result, we finish this proof.
\end{proof}

\subsection{Zeroth-order estimate of $ v$}

Now, we are in a position to give a priori estimates. First, multiplying by $v$ on the first equation of \eqref{eq:CNS1} and integrating over $\Om$, from the Piola identity \eqref{Piola-identity} and boundary conditions, we get the basic energy estimate:
\begin{proposition}\label{basic energy: v}
Assume $v$ is a smooth solution of system \eqref{eq:CNS1} on $[0,T]$. Then, we have
\begin{align*}
\f12\f{d}{dt}(\int_{\Om}\brho|v|^2dx+2\int_{\Om}\brho^2J^{-1}dx)
+\f12\int_{\Om}(\f\mu2|\mathbb{D}^0_{\cA} v|^2+(\lambda+\frac{2}{3}\mu) |\na_\cA\cdot v|^2)\,Jdx=0.
\end{align*}

\end{proposition}

 \bigskip

%

\subsection{First-order estimate of $ v$}
Here, to get the higher regularity of the $v$. We multiply $\pa_t v$ on the both sides of \eqref{eq:CNS1} to get that
\begin{proposition}\label{pro2: v}
Assume that  \eqref{assum: energy3} holds and $v$ is a smooth solution of system \eqref{eq:CNS1} on $[0,T]$, then there holds that for $t\in [0, T]$
\begin{align*}
 \f12\f{d}{dt}\int_{\Om} (\f\mu2|\mathbb{D}^0_{\cA} v|^2+(\lambda+\frac{2}{3}\mu) |\na_\cA\cdot v|^2)\,J\,dx+\|\brho^{\f12}\pa_t v\|_{L^2}^2\leq (C_0+t^{\f12}\mathcal{P}(\frak{C}))(\frak{D}(v) \|\na v\|_{L^2}^2 +1).
\end{align*}

\end{proposition}

\begin{proof}
Taking $L^2$ product with $\pa_t v$ to the first equation of \eqref{eq:CNS1} to get that
\begin{align*}
\|\brho^{\f12}\pa_t v\|_{L^2}^2+\int_{\Om}\na_{J\cA}(\brho^2J^{-2})\cdot\pa_t vdx-\int_{\Om}\na_{J\cA}\cdot \mS_{\cA} (v) \cdot\pa_t vdx=0.
\end{align*}
Due to the Piola identity \eqref{Piola-identity} and the boundary condition $\mathrm{S}_{\cA}(v)\cdot \mathcal{N}|_{x_3=1}=0$ and $v|_{x_3=0}=0$, integration by parts yields
\begin{align*}
-\int_{\Om}\na_{J\cA}\cdot \mS_{\cA} (v)\cdot\pa_t vdx=&\int_{\Om}J \mS_{\cA} (v): \pa_t(\na_\cA v)dx-\int_{\Om}J\mS_{\cA}(v):\na_{\pa_t\cA} vdx.
\end{align*}
Since $\mathbb{D}_\cA(v)$ and $(\na_\cA \cdot v)\mathbb{I}$ are  symmetric quantities, it implies that
\begin{equation*}
\begin{split}
&\int_{\Om}J \mS_{\cA} (v): \pa_t(\na_\cA v)dx=\int_{\Om}J (\mu\mathbb{D}_\cA(v)+\la (\na_\cA \cdot v)\mathbb{I}): \pa_t(\na_\cA v)dx\\
&=\f{\mu}{2}\int_{\Om}J \mathbb{D}_\cA(v): \pa_t\mathbb{D}_\cA(v)dx+\la\int_{\Om}J\na_\cA \cdot v~\pa_t(\na_\cA \cdot v)\\
&
=\f12\f{d}{dt}\int_{\Om}J (\f\mu2|\mathbb{D}^0_{\cA} v|^2+(\lambda+\frac{2}{3}\mu) |\na_\cA\cdot v|^2)\,dx-\f12\int_{\Om}\pa_tJ\Big(\f{\mu}{2} |\mathbb{D}_\cA(v)|^2+\la |\na_\cA \cdot v|^2\Big)dx\\
&=\f12\f{d}{dt}\int_{\Om}J\mathbb{S}_\cA (v): \na_{\cA}v\,dx-\f12\int_{\Om}\mathbb{S}_\cA (v): \na_{\cA}v\pa_t Jdx,
\end{split}
\end{equation*}
which gives that
\begin{align*}
&-\int_{\Om}\na_{J\cA}\cdot \mS_{\cA} (v)\cdot\pa_t vdx
=\f12\f{d}{dt}\int_{\Om}J (\f\mu2|\mathbb{D}^0_{\cA} v|^2+(\lambda+\frac{2}{3}\mu) |\na_\cA\cdot v|^2)\,dx\\
&\qquad\qquad\qquad\qquad-\f12\int_{\Om}\mathbb{S}_\cA (v): \na_{\cA}v\pa_t Jdx
-\int_{\Om}J\mS_{\cA}(v):\na_{\pa_t\cA} vdx.
\end{align*}
To estimate the last two terms of right hand of the above equation, we recall that formula \eqref{equ:d-A}-\eqref{equ:d-J}, Lemma \ref{lem: interpolation} and Lemma \ref{lem:assumption} to get that
\beno
\|\pa_tJ,\pa_t \cA\|_{L^\infty}\leq C_0\|\cA\|_{L^\infty}^2\|\na v\|_{L^\infty}\leq (C+t^{\f12}\mathcal{P}(\frak{C}))\frak{D}(v),
\eeno
which implies that
\beno
&&|\int_{\Om}\mathbb{S}_\cA (v): \na_{\cA}v\pa_t Jdx|+|\int_{\Om}J\mS_{\cA}(v):\na_{\pa_t\cA} vdx|\\
&\leq&(C_0+t^{\f12}\mathcal{P}(\frak{C}))\frak{D}(v)\|J\cA\|_{L^\infty}\|\na v\|_{L^2}^2\leq (C_0+t^{\f12}\mathcal{P}(\frak{C}))\frak{D}(v) \|\na v\|_{L^2}^2 .
\eeno

For the pressure term, we notice it contains $\brho^2$. Thus, we have
\beno
\brho^{-\f12}\na_{J\cA}(\brho^2J^{-2})=\brho^{-\f12}\pa_k(J^{-1}\cA^k_i\brho^2)=\brho^{-\f12}\Big( J^{-1}\cA^k_i\pa_k\brho^2+\pa_k(J^{-2})J\cA^k_i\brho^2 \Big),
\eeno
which implies that for all $t\in [0, T]$, we have
\beno
\|\brho^{-\f12}\na_{J\cA}(\brho^2J^{-2})\|_{L^2}
&\leq&  \|\brho^{\f12}\brho'\|_{L^\infty}(C+t^{\f12}\mathcal{P}(\frak{C})) \|\cA\|_{L^2}+(C_0+t^{\f12}\mathcal{P}(\frak{C}))\|ZJ\|_{L^\infty}\|J\cA\|_{L^2}\\
&\leq& C_0+t^{\f12}\mathcal{P}(\frak{C}),
\eeno
where we use Lemma \ref{lem: cA,JcA}.
Thus, by Holder inequality, we have
\beno
\Big|\int_{\Om}\na_{J\cA}(\brho^2J^{-2})~\pa_t vdx\Big|&\leq& \|\brho^{\f12}\pa_t v\|_{L^2}\|\brho^{-\f12}\na_{J\cA}(\brho^2J^{-2})\|_{L^2}\\
&\leq& (C_0+t^{\f12}\mathcal{P}(\frak{C}))\|\brho^{\f12}\pa_t v\|_{L^2}\leq C_0+t^{\f12}\mathcal{P}(\frak{C})+ \f12\|\brho^{\f12}\pa_t v\|_{L^2}^2.
\eeno
By now, we get the desired result.
\end{proof}

\bigskip

\subsection{High-order estimates of $v$ }
In this subsection, we use the conormal derivative to get the regularity of the horizontal direction. The following is our main results of this subsection:
\begin{proposition}\label{pro1: v }
Assume that  \eqref{assum: energy3} holds and $v$ is a smooth solution of system \eqref{eq:CNS1} on $[0,T]$, then it holds that
\begin{align*}
&\f{d}{dt}\|v\|_{X^{12}_{\f12}}^2+\Big(c_0-\d(C_0+t^{\f12}\mathcal{P}(\frak{C}))\Big)\|\na v\|_{X^{12}}^2\\
&\qquad\leq t^{\f12}\mathcal{P}(\frak{C})\frak{D}^2(v)+C_0\|v\|_{X^{12}_\f12}^2+C_0+t^{\f12}\mathcal{P}(\frak{C}),
\end{align*}

\end{proposition}
\begin{proof}
Acting $Z^m$ on the first equation of \eqref{eq:CNS1} and taking $L^2$ inner product with $\d^{2|m|}Z^mv$, then summing  $\sum_{|m|=0}^{12}$ to obtain
\begin{align*}
\f12\f{d}{dt}\|v\|_{X^{12}_{\f12}}^2-\sum_{|m|=0}^{12}\d^{2|m|}\int_{\Om}Z^m\big(\na_{J\cA}\cdot \mS_{\cA}v\big)\cdot Z^m v\,dx= I_1+I_2
\end{align*}
with
\begin{align*}
&I_1=\sum_{|m|=1}^{12}\d^{2|m|} \int_{\Om}[\brho,Z^m]\pa_t v\cdot Z^m v \, dx,\quad\,I_2=-\sum_{|m|=0}^{12}\d^{2|m|}\int_{\Om}Z^m\big(\na_{J\cA}(J^{-2}\brho^2)\big)\cdot Z^m v \,dx.
\end{align*}

{\bf{Estimate of dissipation term.}}
For the dissipation term, by using integration by parts, we split it into three parts:
\begin{align*}
&-\sum_{|m|=0}^{12}\d^{2|m|} \int_{\Om}Z^m\big(\na_{J\cA}\cdot \mS_{\cA}v\big)\cdot Z^m vdx\\
&=\sum_{|m|=0}^{12}\d^{2|m|} \int_\Om J\mathbb{S}_\cA(Z^mv): \na_{\cA}Z^m v \,dx+\sum_{|m|=1}^{12}\d^{2|m|} \int_{\Om}[Z^m,\mS_{\cA}]v: \na_{J\cA}(Z^m v) dx\\
 &-\sum_{|m|=1}^{12}\d^{2|m|} \bigg(\int_{x_3=1} \mathcal{N}\cdot Z_h^m\mS_{\cA}v\cdot Z_h^m v dS+ \int_{\Om}[Z^m,\na_{J\cA}]\cdot \mS_{\cA}v \cdot Z^m v \,dx\bigg)=:I_3+I_4+I_5.
\end{align*}

\medskip

Next, we deal with the commutators $I_3, I_4$ and $I_5$.

$\bullet$ \underline{Estimates of $I_3$}. Thanks to Lemma \ref{lem: equal na v}, one can see that for any $m:|m|=0, 1,...,12$
\begin{equation*}\label{eq: commutator1-1}
\begin{split}
&\int_\Om J\mathbb{S}_\cA(Z^mv): \na_{\cA}Z^m v \,dx\\
&=\int_{\Om} (\f\mu4|\mathbb{D}^0_{\cA} Z^mv|^2+\frac{\lambda+\frac{2}{3}\mu}{2} |\na_\cA\cdot Z^mv|^2)\,J\,dx\\
&\geq c_1\|\mathbb{D}^0_{\cA} Z^mv\|_{L^2}^2 \geq  c_1\bigg((c_0- t^2 \mathcal{P}(\frak{C}))\|\na (Z^mv)\|_{L^2}^2-C_0\|\brho^\f12 Z^mv\|_{L^2}^2\bigg),
\end{split}
\end{equation*}
which implies
\begin{equation}\label{eq: commutator1-2}
\begin{split}
&\sum_{|m|=0}^{12}\d^{2|m|}\int_\Om J\mathbb{S}_\cA(Z^mv): \na_{\cA}Z^m v \,dx\geq  \sum_{|m|=0}^{12}\d^{2|m|}c_1\bigg(\frac{1}{2}(c_0- t^2 \mathcal{P}(\frak{C}))\|Z^m\nabla\,v\|_{L^2}^2\\
&\qquad\qquad\qquad\qquad\qquad\qquad-\frac{1}{2}(c_0+ t^2 \mathcal{P}(\frak{C}))\|[\na,Z^m]v\|_{L^2}^2-C_0\|\brho^\f12 Z^mv\|_{L^2}^2\bigg),
\end{split}
\end{equation}
For $|m|\geq1$, by direct calculation, we have
\begin{align}\label{eq: commutator1ab}
[\na,Z^m]=m\na \brho Z^{m-1}\pa_{3},
\end{align}
which implies that
\begin{equation}\label{eq: commutator1-3}
\sum_{|m|=1}^{12}\d^{2|m|}\frac{c_1}{2}(c_0+ t^2 \mathcal{P}(\frak{C}))\|[\na,Z^m]v\|_{L^2}^2\leq (C_0+ t^2 \mathcal{P}(\frak{C}))\d^2\|\na v\|_{X^{11}}^2.
\end{equation}
Plugging \eqref{eq: commutator1-3} into \eqref{eq: commutator1-2} shows
\begin{equation*}\label{eq: commutator1-4}
\begin{split}
&\sum_{|m|=0}^{12}\d^{2|m|}\int_\Om J\mathbb{S}_\cA(Z^mv): \na_{\cA}Z^m v \,dx\\
&\geq (c_2- t^2 \mathcal{P}(\frak{C}))\|\nabla\,v\|_{X^{12}}^2-(C_0+ t^2 \mathcal{P}(\frak{C}))\d^2\|\na v\|_{X^{11}}^2-C_0\| v\|_{X^{12}_{\f12}}^2.
\end{split}
\end{equation*}

$\bullet$ \underline{Estimates of $I_4$}. For $|m|\geq 1,$ by a direct calculation, we have
\begin{equation*}\label{eq: commutator2-1}
\begin{split}
&[Z^m,\mD_{\cA}]v= Z^m\Big(\cA_{i}^k\pa_kv_j+\cA_{j}^k\pa_kv_i   \Big)-\Big(\cA_{i}^k\pa_k(Z^mv_j)+\cA_{j}^k\pa_k(Z^mv_i)   \Big)\\
&=\cA_{i}^k[Z^m,\pa_k]v_j+\cA_{j}^k[Z^m,\pa_k]v_i+\sum_{\substack{|m_1|+|m_2|=|m|,\\|m_1|\geq1}}(Z^{m_1} \cA_{i}^kZ^{m_2}\pa_kv_j+Z^{m_1}\cA_{j}^kZ^{m_2}\pa_kv_i)\\
&=m\pa_k \brho\cA_{i}^3Z^{m-1}\pa_3v_j+m\pa_k \brho \cA_{j}^3Z^{m-1}\pa_3v_i\\
&\quad+\sum_{\substack{\substack{|m_1|+|m_2|=|m|,\\|m_1|\geq1}}}(Z^{m_1} \cA_{i}^kZ^{m_2}\pa_kv_j+Z^{m_1}\cA_{j}^kZ^{m_2}\pa_kv_i).
\end{split}
\end{equation*}
By Lemmas \ref{lem:assumption}, \ref{lem:high nav cA}, we have
 \beno
 \d^{|m|}\|[Z^m,\mD_{\cA}]v\|_{L^2}&\leq & \d\Big((C_0+t^{\f12}\mathcal{P}(\frak{C}))\|\na v\|_{X^{11}}+C_0\|\na v\|_{X^{11}}+t^\f12\mathcal{P}(\frak{C})\frak{ D}(v)\Big)\\
 &\leq&\d\Big( C_0\|\na v\|_{X^{11}}+t^\f12\mathcal{P}(\frak{C})\frak{ D}(v)\Big).
  \eeno
  By the same argument, we have
   \beno
\d^{|m|}\|[Z^m,\dv_{\cA}]v\|_{L^2}\leq \d\Big( C_0\|\na v\|_{X^{11}}+t^\f12\mathcal{P}(\frak{C})\frak{ D}(v)\Big).
  \eeno

Combining the above two estimates, we have
\begin{align*}
I_4\leq&  \d(C_0+t^{\f12}\mathcal{P}(\frak{C}))(\d \|\na v\|_{X^{11}}+\|\na v\|_{X^{12}})(C_0\|\na v\|_{X^{11}}+t^\f12\mathcal{P}(\frak{C})\frak{ D}(v)).\\
\leq& \d(C_0+t^{\f12}\mathcal{P}(\frak{C}))\|\na v\|_{X^{12}}^2+t^{\f12}\mathcal{P}(\frak{C})\frak{ D}(v)^2.
\end{align*}

$\bullet$ \underline{Estimates of $I_5$}. A direct calculation gives that
 \begin{align}
 [Z^m,\na_{J\cA}]\cdot \mS_{\cA }v=&Z^m(J\cA_i^k\pa_k(\mS_{\cA} v)^i)-\pa_k(J\cA_i^k(Z^m(\mS_{\cA} v))^i)\nonumber\\
 =&\pa_k\Big(Z^m(J\cA_i^k(\mS_{\cA} v)^i)-J\cA_i^k(Z^m(\mS_{\cA} v))^i\Big)+[Z^m,\pa_k](J\cA_i^k(\mS_{\cA} v)^i).\label{equ: I_5}
  \end{align}
 For the commutator term, we see
\begin{align}\label{est: commutator11}
[Z^m,\pa_3]=-m\pa_3\brho Z^{m-1}\pa_3\sim Z^{m-1}\pa_3,~[Z^m,\pa_h]=-m\pa_h\brho Z^{m-1}\pa_3\sim  Z^{m},
\end{align}
where we used \eqref{assum: brho2}.
Then we have
\begin{align*}
&\Big|\int_{\Om}[Z^m,\pa_k](J\cA_i^k(\mS_{\cA} v)^i)\cdot Z^m v dx\Big|\\
&\leq C_0\Big|\int_{\Om}Z^{m-1}\pa_3(J\cA_i^k(\mS_{\cA} v)^i)\cdot Z_3 Z^{m-1} v dx\Big|\\
&\quad+C_0\Big|\int_{\Om}Z^{m}(J\cA_i^k(\mS_{\cA} v)^i)\cdot Z^m v dx\Big|\leq C_0\Big|\int_{\Om} Z^{m}(J\cA_i^k(\mS_{\cA} v)^i)\cdot Z^{m-1}\na vdx\Big|,
\end{align*}
which combining with  Lemma \ref{lem:high nav cA} follows
\begin{align*}
&\Big|\sum_{|m|=1}^{12}\d^{2|m|}\int_{\Om}[Z^m,\pa_k](J\cA_i^k(\mS_{\cA} v)^i)\cdot Z^m v dx\Big|\\
&\leq
\d ( C_0\|\na v\|_{X^{12}}+t^\f12\mathcal{P}(\frak{C})\frak{ D}(v)) \|\na v\|_{X^{11}}.
\end{align*}
Now, we deal with the first term of the right hand of \eqref{equ: I_5}. By using integration by parts, one has
\begin{align*}
&\sum_{|m|=1}^{12}\d^{2|m|} \int_{\Om}\pa_k\Big(Z^m(J\cA_i^k(\mS_{\cA} v)^i)-J\cA_i^k(Z^m(\mS_{\cA} v))^i\Big)\cdot Z^m v dx\\
&=-\sum_{|m|=1}^{12}\d^{2|m|} \int_{\Om}\Big(Z^m(J\cA_i^k(\mS_{\cA} v)^i)-J\cA_i^k(Z^m(\mS_{\cA} v))^i\Big)\cdot \pa_kZ^m v dx\\
&\quad+\sum_{|m|=1}^{12}\d^{2|m|}\int_{x_3=1}\Big(Z_h^m(J\cA_i^3e_3(\mS_{\cA} v)^i)-J\cA_i^3e_3(Z_h^m(\mS_{\cA} v))^i\Big)\cdot Z_h^m vdS.
\end{align*}
Because of $\mS_{\cA} (v) \mathcal{N}=0$ on the boundary $\{x_3=1\}$, $J\cA_i^3e_3= \mathcal{N}$, and $Z_h^m(\mS_{\cA} v \mathcal{N})=0$ on $\{x_3=1\}$, the second term on the above equality plus the second term of $I_5$ is zero:
\begin{align*}
\sum_{|m|=1}^{12}\d^{2|m|}\int_{x_3=1}&\Big(Z_h^m(\mathcal{N}(\mS_{\cA} v))-\mathcal{N}(Z_h^m(\mS_{\cA} v))\Big)~Z_h^m vdS\\
&+\sum_{|m|=1}^{12}\d^{2|m|} \int_{x_3=1}\mathcal{N}Z_h^m\mS_{\cA}v~Z_h^m v dS=0.
\end{align*}
Hence, all we left is to deal with the commutator
\beno
\int_{\Om}\Big(Z^m(J\cA_i^k(\mS_{\cA} v)^i)-J\cA_i^k(Z^m(\mS_{\cA} v))^i\Big)\cdot \pa_kZ^m v dx.
\eeno
By the same arguments as $I_4$ and using Lemma \ref{lem: interpolation}-\ref{lem:high nav cA}, we deduce that
 \begin{align*}
 \Big| \sum_{|m|=1}^{12}\d^{2|m|}\int_{\Om}&\Big(Z^m(J\cA_i^k(\mS_{\cA} v)^i)-J\cA_i^k(Z^m(\mS_{\cA} v))^i\Big)\cdot \pa_kZ^m vdx  \Big|  \\
 \leq& \d(C_0+t^{\f12}\mathcal{P}(\frak{C}))\|\na v\|_{X^{12}}^2+t^{\f12}\mathcal{P}(\frak{C})\frak{ D}(v)^2.
 \end{align*}
Combining all the above estimates, we get that
\begin{align*}
I_5\leq \d(C_0+t^{\f12}\mathcal{P}(\frak{C}))\|\na v\|_{X^{12}}^2+t^{\f12}\mathcal{P}(\frak{C})\frak{ D}(v)^2.
\end{align*}

So far, we obtain
\begin{equation*}\label{est:disspation}
\begin{split}
&-\sum_{|m|=0}^{12}\d^{2|m|}\int_{\Om}Z^m\big(\na_{J\cA}\cdot \mS_{\cA}v\big)\cdot Z^m vdx\\
&\geq \Big(c_2-\d(C_0+t^{\f12}\mathcal{P}(\frak{C}))\Big)\|\na v\|_{X^{12}}^2-C_0\|v\|_{X^{12}_\f12}^2-t^{\f12}\mathcal{P}(\frak{C})\frak{ D}(v)^2.
\end{split}
\end{equation*}

\medskip

{\bf{Estimate of $I_2.$}} Now, we deal with the pressure.
\begin{align*}
I_2=&\sum_{|m|=0}^{12}\d^{2|m|}\int_{\Om}\pa_kZ^m\big(\cA_i^kJ^{-1}\brho^2\big)\cdot Z^m v^idx\\
&\,+\sum_{|m|=1}^{12}\d^{2|m|}\int_{\Om}[Z^m,\pa_k]\big(\cA_i^kJ^{-1}\brho^2\big)\cdot Z^mv^i\,dx
\triangleq I_{21}+I_{22}.
\end{align*}

$\bullet$ \underline{Estimates of $I_{22}$}.
 Since $Z^m \brho^2\sim \brho^2$ for any $m$, we use \eqref{eq: commutator1ab} and  Lemmas \ref{lem: product2}-\ref{lem: cA,JcA} to get
 \begin{align*}
 I_{22}
 \leq  \d (C_0+t^{\f12}\mathcal{P}(\frak{C}))\|\na v\|_{X^{11}}.
 \end{align*}

 $\bullet$ \underline{Estimates of $I_{21}$}. Because of $\brho|_{x_3=1}=0,$ the boundary terms vanish when we integrate by parts. By the same argument as  $I_5$, it is easy to see $I_{21}$ is bounded by
\begin{align*}
I_{21}&\leq  (C_0+t^{\f12}\mathcal{P}(\frak{C}))(\|\na v\|_{X^{12}}+\d\|\na v\|_{X^{11}}).
\end{align*}

Combining the two estimates, we get
\begin{align*}
I_2\leq &
(C_0+t^{\f12}\mathcal{P}(\frak{C}))(\|\na v\|_{X^{12}}+\d\|\na v\|_{X^{11}}).
\end{align*}

{\bf{Estimate of $I_1.$}}
 For $m\geq1,$ it holds that
 \begin{equation*}\label{eq: commutator2}
 [\brho,Z^m]\sim \sum_{k=0}^{m-1} f_kZ^{k}(\brho\cdot) .
 \end{equation*}
 where $f_k$ are smooth functions which are defined by $\brho$. Thus
\begin{align*}
I_1
&\leq C_0\sum_{\substack{ |m|=1}}^{12}\sum_{k=0}^{m-1}\d^{2|m|} \Big|\int_{\Om}Z^k(\brho\pa_t v)\cdot Z^m vdx\Big|\\
&=C_0\sum_{\substack{ |m|=1}}^{12}\sum_{k=0}^{m-1}\d^{2|m|} \Big|\int_{\Om}Z^k(-\na_{J\cA}(J^{-2}\brho^2)+\na_{J\cA}\cdot \mS_{\cA}v)\cdot Z^m vdx\Big|.
\end{align*}
From the formula above, $I_1$ can be regarded as  lower term to $I_2$ plus dissipation term with the highest order 11. Since $k\leq m-1,$ extra $\d$ is left. Thus, we have
    \begin{align*}
&I_1\leq \d(C_0+t^{\f12}\mathcal{P}(\frak{C}))( \|\na v\|_{X^{12}} +  \d\|\na v\|_{X^{11}})\\
&\qquad\qquad\qquad\qquad+\d (C_0+t^{\f12}\mathcal{P}(\frak{C}))(C_0\|\na v\|_{X^{11}}+t^\f12\mathcal{P}(\frak{C})\frak{ D}(v)) \|\na v\|_{X^{12}}.
\end{align*}
Collecting all estimates together, we finally obtain
\begin{align*}
&\f{d}{dt}\|v\|_{X^{12}_{\f12}}^2+\Big(c_0-\d(C_0+t^{\f12}\mathcal{P}(\frak{C}))\Big)\|\na v\|_{X^{12}}^2\\
&\leq  t^{\f12}\mathcal{P}(\frak{C})\frak{D}^2(v)+C_0\|v\|_{X^{12}_\f12}^2+C_0+t^{\f12}\mathcal{P}(\frak{C}),
\end{align*}
which implies the desired results.
\end{proof}

\bigskip

\subsection{Estimate for $\frak{D}(v)$}
To close the energy estimates, all we left is the estimate of $\frak{D}(v)$ which should be controlled by the energy.

\begin{lemma}\label{lem: relation}
Assume that \eqref{assum: energy3}. Then there exists  $0<T\leq \bar{T}$ and $\d_0>0$  which depend on the initial data, $\sigma_0$ and $\frak{C}$ such that for any $t\in[0,T]$ and $\d\in(0,\d_0)$, it holds that
\begin{align*}
\frak{D}^2(v)\leq C\mathcal{D}(t)+(C+t^{\f12}\mathcal{P}(\frak{C}))(1+t^\f12\frak{D}(v)).
\end{align*}

\end{lemma}
\begin{proof}
Here we only need to control the term $\|\brho^{\kappa-\f12}\tri v\|_{L^2}$. To do that, we go back to the equation of $v$. Since
\begin{align*}
\brho^{-\f12+\kappa}\na_{J_0\cA_0}\cdot \mS_{\cA_0} v=&\brho^{-\f12+\kappa}\na_{J\cA}\cdot \mS_{\cA}v+\brho^{-\f12+\kappa}\big(\na_{J_0\cA_0}\cdot \mS_{\cA_0}  v-\na_{J\cA}\cdot \mS_{\cA}v\big)\\
=&\brho^{-\f12+\kappa}\Big(-\brho \pa_t v-\na_{J\cA}(J^{-2}\brho^2)   \Big)+\brho^{-\f12+\kappa}\big(\na_{J_0\cA_0}\cdot \mS_{\cA_0}  v-\na_{J\cA}\cdot \mS_{\cA}v\big),
\end{align*}
which implies that
\begin{align*}
\|\brho^{-\f12+\kappa}\na_{J_0\cA_0}\cdot \mS_{\cA_0} v\|_{L^2}\leq&\|\brho^{\f12}\pa_t v\|_{L^2}+\|\brho^{-\f12+\kappa}\na_{J\cA}(J^{-2}\brho^2)\|_{L^2}\\
&+\|\brho^{-\f12+\kappa}\na_{J_0\cA_0-J\cA}\cdot \mS_{\cA_0} v\|_{L^2}+\|\brho^{-\f12+\kappa}\na_{J\cA}\cdot \mS_{\cA_0-\cA}v\|_{L^2}\\
\triangleq&\|\brho^{\f12}\pa_t v\|_{L^2}+I_1+I_2+I_3.
\end{align*}

Owing to Lemma \ref{lem: cA,JcA}, we have
\begin{align*}
I_1 \leq C_0+t^{\f12}\mathcal{P}(\frak{C}).
\end{align*}

For $I_2$, by Lemma \ref{lem: hardy}, Lemma \ref{lem: cA,JcA} and \eqref{eq: JcA-I}-\eqref{est: JcA-I}, we have
\begin{align*}
I_2\leq& \|J_0\cA_0-J\cA\|_{L^\infty}\|\brho^{-\f12+\kappa}\na\cdot  \mS_{\cA_0}  v\|_{L^2}\\
\leq&  t^\f12(C_0+t^{\f12}\mathcal{P}(\frak{C}))(\|\brho^{-\f12+\kappa}\na v\|_{L^2}+\|\brho^{-\f12+\kappa}\na^2 v\|_{L^2})\\
\leq&  t^\f12(C_0+t^{\f12}\mathcal{P}(\frak{C}))(\|\brho^{-\f12+\kappa} \tri v\|_{L^2}+\|\brho^{-\f12+\kappa} Z_h \pa_3 v\|_{L^2}+\|\brho^{-\f12+\kappa} Z_h^2 v\|_{L^2}+\frak{D}(v))\\
\leq&
t^\f12(C_0+t^{\f12}\mathcal{P}(\frak{C}))(\|\brho^{-\f12+\kappa} \tri v\|_{L^2}+\|\brho^{\f12+\kappa} Z_h \tri v\|_{L^2}+\|\brho^{\f12+\kappa} Z_h^2 \na v\|_{L^2}+\frak{D}(v))\\
\leq&t^\f12(C+t^{\f12}\mathcal{P}(\frak{C}))\frak{D}(v).
\end{align*}
Similarly, by the fact that
\begin{align*}
\cA-\cA_0=(\cA J-\cA_0 J_0)J^{-1}+J^{-1}(J_0-J)\cA_0
\end{align*}
and
\begin{align*}
J-J_0=\int_0^t \pa_t Jds=\int_0^t J\na_{\cA}vds,
\end{align*}
combine \eqref{eq: JcA-I} with  Lemma \ref{lem:assumption} to get
\begin{align*}
I_3\leq t^\f12 (C+t^{\f12}\mathcal{P}(\frak{C}))\frak{D}(v).
\end{align*}

Collecting all above estimates  to obtain
\begin{align}\label{est: tri v}
\|\brho^{-\f12+\kappa}\na_{J_0\cA_0}\cdot \mS_{\cA_0} v\|_{L^2}\leq& \|\brho^{\f12}\pa_t v\|_{L^2}+(C+t^{\f12}\mathcal{P}(\frak{C}))(1+t^\f12\frak{D}(v)).
\end{align}

 Next, we give the relationship between $\tri v$ and $\na_{J_0\cA_0}\cdot \mS_{\cA_0} v$.  It is easy to find that
  \begin{align*}
 \na_{J_0\cA_0}\cdot \mS_{\cA_0} v
 =\left(\begin{array}{c}\mu J^{-1}_0\pa_3^2 v_1 \\
 \mu J^{-1}_0\pa_3^2 v_2\\
 (2\mu+\la) J^{-1}_0\pa_3^2 v_3
 \end{array}\right)
 +\mbox{some terms likes}~ Z\na v.
 \end{align*}

 %
%

By Lemma \ref{lem: hardy} and interpolation inequality, we have
\begin{align}\label{est: tri v 1}
\|\brho^{-\f12+\kappa} Z_h\na v\|_{L^2}\leq& C_0\|\brho^{\f12+\kappa} Z_h\na v\|_{L^2}+C_0\|\brho^{\f12+\kappa} Z_h\na^2 v\|_{L^2}\\
\nonumber
\leq& C_0\| \na v\|_{L^2_{x_3}(H^2_h)}+C_0\|\brho^{-\f12+\kappa} \tri v\|_{L^2}^\theta \|\brho \tri v\|_{L^2}^{1-\theta}\\
\nonumber
\leq& C_\e\| \na v\|_{L^2_{x_3}(H^2_h)}+\e \|\brho^{-\f12+\kappa} \tri v\|_{L^2},
\end{align}
where we use Young inequality  in the last step and $\theta\in(0,1).$


Taking $\e$ small enough and using \eqref{assum: energy3}, \eqref{est: tri v}, \eqref{est: tri v 1}, we have
\begin{align}\label{est: tri v 2}
\|\brho^{-\f12+\kappa} \tri v\|_{L^2}\leq \|\brho^{\f12}\pa_t v\|_{L^2}+(C+t^{\f12}\mathcal{P}(\frak{C}))(1+t^\f12\frak{D}(v))+\f{C_0}{\d^2}\|\na v\|_{X^{12}}.
\end{align}

Combining \eqref{est: tri v} and \eqref{est: tri v 2}, we obtain the desired results.

\end{proof}

\subsection{Proof of Proposition \ref{pro: main} }

Now, from Proposition \ref{pro2: v}--Proposition \ref{pro1: v }, we obtain that
\begin{equation*}
\begin{split}
&\Big(c_0-\d(C_0+t^{\f12}\mathcal{P}(\frak{C}))\Big)\Big(\sup_{\tau\in[0, t]} \mathcal{E}(\tau) + \int_{0}^t \mathcal{D}(\tau)\Big)\\
&\leq  \Big(c_0-\d(C_0+t^{\f12}\mathcal{P}(\frak{C}))\Big)  \mathcal{E}(0)+t^{\f12}\mathcal{P}(\frak{C})(1+ \sup_{\tau\in[0, t]} \mathcal{E}(\tau)).
\end{split}
\end{equation*}

Now, we give the estimates of $J$. By the definition of $J$, we have
\begin{align*}
J-J_0=\int_0^t \pa_t Jds=\int_0^t J\na_{\cA}vds,
\end{align*}
which implies that
\begin{align*}
|J-J_0|\leq \|J\|_{L^\infty}\|\cA\|_{L^\infty}\|\na v\|_{L^1_tL^\infty}\leq t^\f12(C+t^{\f12}\mathcal{P}(\frak{C})).
\end{align*}
Then by the Lemma \ref{lem: relation} and standard bootstrap argument, it implies the proposition \ref{pro: main} proved.

\renewcommand{\theequation}{\thesection.\arabic{equation}}
\setcounter{equation}{0}


\section{Local well-posedness}
In this section, we will first give existence and uniqueness of strong solutions of system \eqref{eq:CNS1}, which is motivated by the method in \cite{GT2013}. First, we give some definitions of functional spaces. Given $T>0,$ let $\widetilde{Y}_T$ and ${Y}_T$ are defined by
\begin{equation*}
  \begin{split}
   &\widetilde{Y}_T\triangleq  C([0, T], X^{0}_{\f12})\cap\,L^2([0, T],\,H^1),\\
   &Y_T\triangleq \{v\in \widetilde{Y}_T\cap C([0, T], X^{12}_{\f12}\cap\,H^1) \big|\|v\|_{Y_T}<+\infty\},
  \end{split}
\end{equation*}
where
$\|v\|_{\widetilde{Y}_T}:=\sup_{t\in[0,T]}\|\brho^{\f12}v\|_{L^2}^2+\| v\|_{L^2_T H^1}^2$
and
$\|v\|_{Y_T}:=\sup_{t\in[0,T]}(\|v\|_{X^{12}_{\f12}}^2+\|\na v\|_{L^2}^2 )+\|\brho^{-\f12+\kappa}\tri v\|_{L^2_TL^2}^2  +\|\na v\|_{L_T^2X^{12}}^2+\|\brho^{\f12}\pa_t v\|_{L_T^2L^2}^2$.

Now, we define map $\Theta: Y_T\to Y_T$ as follows. For any given $\widetilde{v}\in Y_T,$ $v:=\Theta(\widetilde{v})$ is the solution of the following linear $\mathcal{A}$-equations:
\begin{align}\label{eq:LCNS}
\left\{
\begin{aligned}
&\brho\pa_t v+\na_{\widetilde{J}\widetilde{\cA}}((\widetilde{J})^{-2}\brho^2)-\na_{\widetilde{J}\widetilde{\cA}}\cdot \mS_{\widetilde{\cA}}(v)=0.\quad\mbox{in}\quad \Om,\\
&\mathrm{S}_{\widetilde{\cA}}(v)~ \widetilde{\mathcal{N}}=0,\quad\mbox{on}\quad \Gamma,\\
&v|_{x_3=0}=0,\\
&v|_{t=0}=v_0  \quad\mbox{in}\quad \Om.
\end{aligned}
\right.
\end{align}

\subsection{Existence and uniqueness of the strong solution to \eqref{eq:LCNS}.}

Our aim in this subsection is to construct strong solutions to linear $\mathcal{A}$-equations \eqref{eq:LCNS}.

\begin{lemma}\label{lem: weak}
Assume that $\brho^{\f12} v_0,\na v_0\in L^2$ and $\widetilde{v}\in Y_T,$ then there exists a positive time $T_1 \in (0, T]$ such that the system \eqref{eq:LCNS} has a unique strong solution $v$ with
\begin{align*}
&\brho^{\frac{1}{2}}\,v \in C([0,T_1],L^2),\,v\in C([0,T_1],H^1),\\
&\brho \pa_t v\in L^2(0,T_1; (H^1)^*),\, \brho^{\f12}\pa_t v\in L^2([0,T_1],L^2),\, \brho^{-\f12+\kappa}\tri v\in L^2([0,T_1],L^2).
\end{align*}
Moreover, the solution satisfies the following estimate
\begin{equation*}
  \begin{split}
   &\sup_{t\in[0,T_1]}(\|\brho^{\f12}v\|_{L^2}^2+\|\na v\|_{L^2}^2)+\| v\|_{L^2_{T_1} H^1}^2+\|\brho^{\f12}\pa_t v\|_{L^2_{T_1}L^2}^2\\
   &\qquad\qquad+\|\brho^{-\f12+\kappa}\tri v\|_{L^2_{T_1}L^2}^2+\|\brho \pa_t v\|_{L^2_{T_1}(H^1)^*}^2
\leq C_0\|\brho^{\f12}v_0\|_{L^2}^2+C_0\|\na v_0\|_{L^2}^2+C_0(1+T_1).
  \end{split}
\end{equation*}

\end{lemma}
\begin{proof} We split the proof of the lemma into four steps.

{\it{Step 1: Galerkin approximation.}}
We first use Galerkin method to construct approximate solutions of the system \eqref{eq:LCNS}. Let $\{w_k\}_{k=1}^\infty$ are orthonormal basis of $H^1(\Om)$ which satisfy boundary condition  $\mathrm{S}_{\widetilde{\cA}}(w_k)~ \widetilde{\mathcal{N}}|_{x_3=1}=0$ and $w_k|_{x_3=0}=0$  and set approximate solution with the form
\begin{align*}
v^m(t,x):=\sum_{k=1}^m d^m_k(t)w_k(x), \quad d^m_k(t)\quad\text{will be determined later on},
\end{align*}
which solves the linear system
\begin{align}\label{eq:LCNS-appr}
\left\{
\begin{aligned}
&\brho\pa_t v^m+\na_{\widetilde{J}\widetilde{\cA}}((\widetilde{J})^{-2}\brho^2)-\na_{\widetilde{J}\widetilde{\cA}}\cdot \mS_{\widetilde{\cA}}(v^m)=0.\quad\mbox{in}\quad \Om,\\
&\mathrm{S}_{\widetilde{\cA}}(v^m)~ \widetilde{\mathcal{N}}=0,\quad\mbox{on}\quad \Gamma,\\
&v^m|_{x_3=0}=0,\\
&v^m|_{t=0}=v_0^m=\sum_{k=1}^m d^m_k(0)w_k(x)  \quad\mbox{in}\quad \Om
\end{aligned}
\right.
\end{align}
in the sense of the distribution, where $d^m_k(0)=\int_{\Om} v_0 w_k$ for $k=1, ..., m$.

Taking the test function $\phi=w_\ell$, $\ell=1,\cdots, m$, from the weak formula of the system \eqref{eq:LCNS-appr}, we obtain the following ordinary differential equations
\begin{equation}\label{eq:ODE}
\begin{cases}
&\sum_{k=1}^m\int_{\Om}\brho w_k w_\ell dx~ d^m_k(t)'+\sum_{k=1}^m\int_\Om \mS_{\widetilde{\cA}}w_k:\na_{\widetilde{J}\widetilde{\cA}}w_\ell dx~d^m_k(t)=\int_\Om \brho^2\widetilde{J}^{-2}\na_{\widetilde{J}\widetilde{\cA}}\cdot w_\ell dx,\\
&d^m_k(0)=\int_{\Om} v_0 w_k.
\end{cases}
\end{equation}
Notice that the matrix $\Big(\int_{\Om}\brho w_k w_\ell dx \Big)_{m\times m}$ is invertible for any $m\geq1$, and the coefficient $\int_\Om \mS_{\widetilde{\cA}}w_k:\na_{\widetilde{J}\widetilde{\cA}}w_\ell dx$ (in front of $d^m_k(t)$) is continuous in terms of $t\in[0,T]$ because of $\widetilde{v}\in Y_T$, we know that \eqref{eq:ODE} is a non-generate linear ODE system with continuous coefficients. Due to the classical theory of ODE, we find  solutions $d^m_k(t)\in C^1([0,T]), k=1,\cdots, m$, which means approximate solutions $v^m(t,x)$ exist and belong to the space $C^1([0,T],H^1(\Om)).$
\medskip

{\it{Step 2: Uniform estimates for $v^m.$}}
 Multiplying $d^m_\ell(t)$ on the both sides of \eqref{eq:ODE} and taking the summation in terms of $\ell=1,\cdots, m$, one has
\begin{align*}
\int_{\Om}\brho\pa_t v^m \cdot \,v^m+\int_{\Om}\mS_{\widetilde{\cA}}v^m:\na_{\widetilde{J}\widetilde{\cA}}v^m dx=\int_\Om \brho^2\widetilde{J}^{-2}\na_{\widetilde{J}\widetilde{\cA}}\cdot v^m dx.
\end{align*}
Then Lemma \ref{lem: cA,JcA} and Lemma \ref{Lem: H^1} give that
\begin{align}\label{est: v^m}
\f12\f{d}{dt}\|\brho^{\f12}v^m\|_{L^2}^2+c_0\| v^m\|_{H^1}^2\leq& \int_\Om \brho^2\widetilde{J}^{-2}\na_{\widetilde{J}\widetilde{\cA}}\cdot v^m dx+C_0\|\brho^{\f12}v^m\|_{L^2}^2\\
\nonumber
\leq& C_0\|\na v^m\|_{L^2}+C_0\|\brho^{\f12}v^m\|_{L^2}^2,
\end{align}
for $t$ small enough.

By Gronwall's inequality, we know there exists $T_1>0$ independent of $m$ such that
\begin{align}\label{est: low 1}
\sup_{t\in[0,T_1]}\|\brho^{\f12}v^m\|_{L^2}^2+\int_0^{T_1}\| v^m\|_{H^1}^2ds\leq C_0\|\brho^{\f12}v^m_0\|_{L^2}^2+C_0T_1.
\end{align}

For any test function $\phi \in  C([0, T], H^1)$ with $\phi|_{x_3=0}=0$ and $\|\phi\|_{L^2_TH^1}\leq1,$ owing to the weak formula of the system \eqref{eq:LCNS-appr}, we deduce from \eqref{est: low 1} that
\begin{align*}
|\int_0^{T_1}\langle\brho\pa_t v^m, \,\phi \,\rangle\,ds| &=|-\int_0^{T_1}\int_{\Om}\mS_{\widetilde{\cA}}v^m:\na_{\widetilde{J}\widetilde{\cA}}\phi  \,dxds+\int_0^{T_1}\int_\Om \brho^2\widetilde{J}^{-2}\na_{\widetilde{J}\widetilde{\cA}}\cdot \phi \,dxds|\\
&\leq\,(C_0+C_0\|\na v^m\|_{L^2_{T_1}L^2})\|\phi\|_{L^2_TH^1}\leq (C_0(1+T_1^\f12)+C_0\|\brho^{\f12}v^m_0\|_{L^2})\|\phi\|_{L^2_TH^1},
\end{align*}
which follows from the dual argument that
\begin{align}\label{est: dual}
\|\brho \pa_t v^m\|_{L^2_{T_1}(H^1)^*}\leq C_0(1+T_1^\f12)+C_0\|\brho^{\f12}v^m_0\|_{L^2}.
\end{align}

 Multiplying $d^m_\ell(t)'$ on the both sides of \eqref{eq:ODE} and taking the summation in terms of $\ell=1,\cdots, m$, we have
\begin{align*}
\int_{\Om}\brho|\pa_t v^m|^2+\int_{\Om}\mS_{\widetilde{\cA}}v^m:\na_{\widetilde{J}\widetilde{\cA}}\pa_tv^m dx=\int_\Om \brho^2\widetilde{J}^{-2}\na_{\widetilde{J}\widetilde{\cA}}\cdot \pa_tv^m dx.
\end{align*}
Similar estimate in Proposition \ref{pro2: v} implies that
\begin{align*}
\f12\f{d}{dt}&\int_{\Om}\widetilde{J}\mathbb{S}_{\widetilde{\cA}} v^m: \na_{\widetilde{\cA}}v^mdx+\|\brho^{\f12}\pa_t v^m\|_{L^2}^2\\
\leq& \Big| \f12\int_{\Om}\mathbb{S}_{\widetilde{\cA}} v^m: \na_{\widetilde{\cA}}v^m\pa_t \widetilde{J}dx\Big| +\Big| \int_{\Om}\widetilde{J}\mS_{\widetilde{\cA}}v^m:\na_{\pa_t{\widetilde{\cA}}} v^mdx \Big|+\Big|\int_\Om \brho^2\widetilde{J}^{-2}\na_{\widetilde{J}\widetilde{\cA}}\cdot \pa_tv^m dx\Big|.
\end{align*}
 Since
 $\widetilde{v}\in Y_T$, we infer that
\begin{align*}
 \Big| \f12\int_{\Om}\mathbb{S}_{\widetilde{\cA}} v^m: \na_{\widetilde{\cA}}v^m\pa_t \widetilde{J}dx\Big| +\Big| \int_{\Om}\widetilde{J}\mS_{\widetilde{\cA}}v^m:\na_{\pa_t{\widetilde{\cA}}} v^mdx \Big|\leq C\|\na  v^m\|_{L^2}^2\frak{D}(\widetilde{v})\end{align*}
and
\begin{align*}
\Big|\int_\Om \brho^2\widetilde{J}^{-2}\na_{\widetilde{J}\widetilde{\cA}}\cdot \pa_tv^m dx\Big|\leq C_0\|\brho^{\f12}\pa_t v^m\|_{L^2}.
\end{align*}
As a result, we get
\begin{align*}
\f12\f{d}{dt}\int_{\Om}\widetilde{J}\mathbb{S}_{\widetilde{\cA}} v^m: \na_{\widetilde{\cA}}v^mdx+\|\brho^{\f12}\pa_t v^m\|_{L^2}^2\leq& C\|\na  v^m\|_{L^2}^2\frak{D}(\widetilde{v})+C_0\|\brho^{\f12}\pa_t v^m\|_{L^2}.
\end{align*}
Integrating time from $0$ to $T_1$ and using $\widetilde{v}\in Y_T$ and Lemma \ref{lem: relation}, we obtain
\begin{equation*}
  \begin{split}
&\sup_{t\in[0,T_1]}\|\na v^m\|_{L^2}^2+\|\brho^{\f12}\pa_t v^m\|_{L^2_{T_1}L^2}^2\\
&\leq C_0\|\na v_0^m\|_{L^2}^2+C_0T_1^\f12\sup_{t\in[0,T_1]}\|\na v^m\|_{L^2}^2(\int_0^{T_1}\frak{D}(\widetilde{v})^2ds)^{\f12}+C_0T_1.
  \end{split}
\end{equation*}
Taking $T_1$ small enough such that the second term on the right hand side absorbed by the left hand side, we obtain
\begin{align}\label{est: low2}
\sup_{t\in[0,T_1]}&\|\na v^m\|_{L^2}^2+c_0\|\brho^{\f12}\pa_t v^m\|_{L^2_{T_1}L^2}^2\leq 2C_0\|\na v_0^m\|_{L^2}^2+C_0T_1.
\end{align}
Combining estimate \eqref{est: low 1}, \eqref{est: dual} and \eqref{est: low2} together, there holds that
\begin{align}\label{est: low3}
\sup_{t\in[0,T_1]}&(\|\brho^{\f12}v^m\|_{L^2}^2+\|\na v^m\|_{L^2}^2)+\| v^m\|_{L_{T_1}^2 H^1}^2+\|\brho^{\f12}\pa_t v^m\|_{L^2_{T_1}L^2}^2+\|\brho \pa_t v^m\|_{L^2_{T_1}(H^1)^*}^2\\
\nonumber
\leq& 2C_0\|\brho^{\f12}v^m_0\|_{L^2}^2+2C_0\|\na v_0^m\|_{L^2}^2+C_0(1+T_1).
\end{align}
{\it{Step 3: Passing to the limit.}}
Since $$\sup_{t\in[0,T_1]}(\|\brho^{\f12}v^m\|_{L^2}^2+\|\na v^m\|_{L^2}^2)+\| v^m\|_{L^2_{T_1} H^1}^2+\|\rho^{\f12}\pa_t v^m\|_{L^2_{T_1}L^2}^2+\|\brho \pa_t v^m\|_{L^2_{T_1}(H^1)^*}^2$$ is uniformly bounded, up to the extraction of a subsequence, we know as $m\to \infty$
\begin{align}\label{weak-limit-1}
\left\{
\begin{aligned}
&\brho^{\f12}v^m\rightharpoonup^* \brho^{\f12}v\quad\mbox{in $L^\infty_{T_1} L^2$},\\
&\na v^m\rightharpoonup^*\na v\quad\mbox{in $L^\infty_{T_1} L^2$},\\
&\brho \pa_t v^m \rightharpoonup \brho\partial _t v \quad\mbox{in $L^2_{T_1} (H^1)^*$},\\
& v^m \rightharpoonup v \quad\mbox{in $L^2_{T_1} H^1$}.
\end{aligned}
\right.
\end{align}
By lower semicontinuity and energy estimate \eqref{est: low3}, we use the fact $\|v^m(0)-v_0\|_{L^2(\Om)}\to 0$ as $m\to \infty$ to infer that
\begin{equation}\label{eneygy-H1-1}
\begin{split}
\sup_{t\in[0,T_1]}&(\|\brho^{\f12}v\|_{L^2}^2+\|\na v\|_{L^2}^2)+\| v\|_{L^2_{T_1} H^1}^2+\|\rho^{\f12}\pa_t v\|_{L^2_{T_1}L^2}+\|\brho \pa_t v\|_{L^2_{T_1}(H^1)^*}^2\\
\leq&  4C_0\|\brho^{\f12}v_0\|_{L^2}^2+4C_0\|\na v_0\|_{L^2}^2+C_0(1+T_1),
\end{split}
\end{equation}
and $v$ is a weak solution to the linear $\mathcal{A}$-equations \eqref{eq:LCNS}. Moreover, according to \eqref{eneygy-H1-1},  we may obtain from Aubin-Lions's lemma \cite{Simon1990} that $v \in C([0, T_1], X^0_{\frac{1}{2}} \cap\, H^1)$.

{\it{Step 4: The strong solution.}}
Now, we prove the above weak solution $v$ is a strong one. In fact, for a.e $t\in[0,T],$ $v(t)$ is a weak solution to the elliptic system in the sense of
\begin{align}
\int_\Om \mS_{\widetilde{\cA}}v:\na_{\widetilde{J}\widetilde{\cA}}\phi dx=\int_\Om \Big(\na_{\widetilde{J}\widetilde{\cA}}(\brho^2\widetilde{J}^{-2})-\brho\pa_t v\Big) \phi dx
\end{align}
for $\phi\in H^1. $ Since $\brho^{-\f12}\Big(\na_{\widetilde{J}\widetilde{\cA}}(\brho^2\widetilde{J}^{-2})-\brho\pa_t v\Big)\in L^2$ for a.e $t\in[0,T],$ by elliptic regularity theory, we know this system admires a strong solution $v$ solving \eqref{eq:LCNS}  with $\rho^{-\f12+\kappa}\tri v\in L^2([0,T],L^2)$. The uniqueness comes from energy estimates with zero initial data.
\end{proof}

\subsection{High regularity of $v$}
In this subsection, we prove when $\widetilde{v}\in Y_T,$ so does $v:=\Theta (\widetilde{v}).$ It is mainly based on the priori estimates in Section 3.
\begin{lemma}\label{lem: regular}
Assume that $v$ is a strong solution obtained in Lemma \ref{lem: weak} and $\widetilde{v}\in Y_{T}$ with  initial data $v_0\in Y_0$, then we have $v\in Y_{T_1}$ and satisfies
\begin{align*}
\|v\|_{Y_{T_1}}\leq CT_1+C_0\|v_0\|_{Y_0},
\end{align*}
where the constant $C$ depends on $\|\widetilde{v}\|_{Y_T}$.

\end{lemma}
\begin{proof}
 We take  $\widetilde{\cA},\widetilde{J}$ instead of $\cA, J$ respectively in those estimates in Proposition \ref{pro2: v} and Proposition \ref{pro1: v }. System \eqref{eq:LCNS} is a linear system due to $\widetilde{\cA},\widetilde{J}$ are regarded as known quantities, so for small $T_1>0$, it is easy to arrive at the following estimate:
\begin{align*}
\|v^m\|_{Y_{T_1}}\leq CT_1+C_0\|v^m_0\|_{Y_0}.
\end{align*}
Passing to the limit, we get the desired results.

 \end{proof}

 \begin{remark}
By Lemma \ref{lem: regular}, we know that $\Theta : Y_{T_1}\to Y_{T_1} $ is well-defined.
\end{remark}

\subsection{Contraction }
By Lemma \ref{lem: weak} and Lemma \ref{lem: regular},  we know that if $\widetilde{v} \in Y_T$ with $T>0$ sufficiently small , we can find a unique strong solution of equation \eqref{eq:LCNS} with regular $v=\Theta(\widetilde{v})\in Y_T.$ In order to construct the solution to \eqref{eq:CNS1}, we need to construct approximate solutions. The approximate solutions $\{\xi^{(n)},\,v^{(n)}\}_{n=1}^{\infty}$ we defined are iterated as follows:
\begin{align}\label{eq: v^(n)}
\left\{
\begin{aligned}
&\partial_t\xi^{(n)}=v^{(n)}\quad\mbox{in}\quad \Om,\\
&\brho\pa_t v^{(n)}+\na_{J^{(n-1)}\cA^{(n-1)}}((J^{(n-1)})^{-2}\brho^2)-\na_{J^{(n-1)}\cA^{(n-1)}}\cdot \mS_{\cA^{(n-1)}}v^{(n)}=0\quad\mbox{in}\quad \Om,\\
&\mathrm{S}_{\cA^{(n-1)}}v^{(n)}~ N^{(n-1)}=0,\quad\mbox{on}\quad \Gamma,\\
&v^{(n)}|_{x_3=0}=0,\\
&(\xi^{(n)},\,v^{(n)}|_{t=0}=(\xi_0,\,v_0)  \quad\mbox{in}\quad \Om.
\end{aligned}
\right.
\end{align}
 with $\{\xi^{(1)},\,v^{(1)}\}$ be the solution of linear  equation
\begin{align}\label{eq: v^(1)}
\left\{
\begin{aligned}
&\partial_t\xi^{(1)}=v^{(1)}\quad\mbox{in}\quad \Om,\\
&\brho\pa_t v^{(1)}+\na_{J_0\cA_0} (\brho^2J_0^{-1})-\na_{J_0\cA_0}\cdot \mS_{\cA_0} v^{(1)}=0\quad\mbox{in}\quad \Om,\\
&\mathrm{S}_{\cA_0} v^{(1)}~ \mathcal{N}_0=0,\quad\mbox{on}\quad \Gamma,\\
&v^{(1)}|_{x_3=0}=0,\\
&(\xi^{(1)},\,v^{(1)}|_{t=0}=(\xi_0,\,v_0)  \quad\mbox{in}\quad \Om,
\end{aligned}
\right.
\end{align}
where $\cA_0,J_0$ are given by $\eta_0(x)= x+\xi_0(x)$ and $\mathcal{N}_0=\pa_1\eta_0\times \pa_2\eta_0$ on $\{x_3=1\}.$
Since \eqref{eq: v^(n)} is a decouple linear system in terms of $\xi^{(n)}$ and $v^{(n)}$, we need only to solve first $v^{(n)}$ then $\xi^{(n)}$ according to the first equation in \eqref{eq: v^(n)}. Notice that \eqref{eq: v^(1)} is linear, the assumption on initial data $
\|v_0\|_{Y_0}^2:=\|v_0\|_{X^{12}_{\f12}}^2+\|\na v_0\|_{L^2}^2 \leq \f{M}{2C_0}$ guarantees that $v^{(1)}\in Y_T$ with bound $\|v^{(1)}\|_{Y_T}^2\leq M.$ By Lemma \ref{lem: regular}, we obtain $\{v^{(n)}\}_{n=1}^{\infty}\subset Y_T$ for any $n\geq1.$

%
Next, our goal in this subsection is to prove sequence $\{v^{(n)}\}_{n=1}^{\infty}$  is contracted under norm $\widetilde{Y}_T$.


 First of all, we deduce $\sigma(v^{(n)})\triangleq v^{(n+1)}-v^{(n)}$ satisfies the following equation
\begin{align}\label{eq: v^(n+1)-v^n}
\left\{
\begin{aligned}
\brho\pa_t& \sigma(v^{(n)})-\Big(\na_{J^{(n)}\cA^{(n)}}\cdot \mS_{\cA^{(n)}}v^{(n+1)}-\na_{J^{(n-1)}\cA^{(n-1)}}\cdot \mS_{\cA^{(n-1)}}v^{(n)}\Big)\\
&+\Big( \na_{J^{(n)}\cA^{(n)}}\big((J^{(n)})^{-2}\brho^2\big)-\na_{J^{(n-1)}\cA^{(n-1)}}\big((J^{(n-1)})^{-2}\brho^2\big)\Big)
=0\quad\mbox{in}\quad \Om,\\
&\mathrm{S}_{\cA^{(n)}}v^{(n+1)}~ \mathcal{N}^{(n)}-\mathrm{S}_{\cA^{(n-1)}}v^{(n)}~ \mathcal{N}^{(n-1)}=0,\quad\mbox{on}\quad \Gamma,\\
&\sigma(v^{(n)})|_{x_3=0}=0,\\
&\sigma(v^{(n)})|_{t=0}=v^{(n+1)}_0-v^{(n)}_0 =0\quad\mbox{in}\quad \Om.
\end{aligned}
\right.
\end{align}

\begin{lemma}\label{lem1: v^(n+1)-v^n}
Assume that $\{v^{(n)}\}_{n=1}^{\infty}$ be the solutions of equation \eqref{eq: v^(n)} with bound $\|v^{(n)}\|_{Y_T}^2\leq M$ for each $n\geq1.$ It holds that
\begin{align*}
\f{d}{dt}&\|\brho^{\f12}\sigma(v^{(n)})\|_{L^2}^2+\|\sigma(v^{(n)})\|_{H^1}^2
\leq Ct\|\sigma(v^{(n-1)})\|_{L_t^2L^2}^2(1+\frak{ D}(\sigma(v^{(n)}))^2).
\end{align*}
Moreover, taking $T$ small enough, the sequence  $v^{(n)}$ is a Cauchy sequence in the space $\widetilde{Y}_T$.

\end{lemma}
\begin{proof}
Taking $L^2$ inner product between \eqref{eq: v^(n+1)-v^n} and $\sigma(v^{(n)})$, we obtain
\begin{align*}
\f12\f{d}{dt}&\|\brho^{\f12}\sigma(v^{(n)})\|_{L^2}^2-\int_{\Om}\Big(\na_{J^{(n)}\cA^{(n)}}\cdot \mS_{\cA^{(n)}}v^{(n+1)}-\na_{J^{(n-1)}\cA^{(n-1)}}\cdot \mS_{\cA^{(n-1)}}v^{(n)}\Big)~ \sigma(v^{(n)})dx\\
=
&-\int_{\Om}\Big( \na_{J^{(n)}\cA^{(n)}}\big((J^{(n)})^{-2}\brho^2\big)-\na_{J^{(n-1)}\cA^{(n-1)}}\big((J^{(n-1)})^{-2}\brho^2\big)\Big)
~ \sigma(v^{(n)})dx.
\end{align*}

{\bf{Estimate of dissipation term.}}
Since
$$e_3J^{(n)}(\cA^{(n)})^3_i=\mathcal{N}^{(n)},\quad \,e_3J^{(n-1)}(\cA^{(n-1)})^3_i=\mathcal{N}^{(n-1)}$$
and
$$\mathrm{S}_{\cA^{(n)}}(v^{(n+1)})~ \mathcal{N}^{(n)}-\mathrm{S}_{\cA^{(n-1)}}(v^{(n)})~ \mathcal{N}^{(n-1)}=0 \quad \mbox{on}\quad  \Gamma,$$
we get by using integration by parts that
\begin{align*}
&-\int_{\Om}\Big(\na_{J^{(n)}\cA^{(n)}}\cdot \mS_{\cA^{(n)}}v^{(n+1)}-\na_{J^{(n-1)}\cA^{(n-1)}}\cdot \mS_{\cA^{(n-1)}}v^{(n)}\Big)~ \sigma(v^{(n)})dx\\
=&\int_{\Om}\Big(J^{(n)}(\cA^{(n)})^k_i(\mS_{\cA^{(n)}}v^{(n+1)})^i_l-J^{(n-1)}(\cA^{(n-1)})^k_i(\mS_{\cA^{(n-1)}}v^{(n)})^i_l\Big)~ \pa_k\sigma(v^{(n)}_l)dx.\\
=&\int_{\Om} J^{(n)}\cA^{(n)}\mS_{\cA^{(n)}}\sigma(v^{(n)}) \cdot\pa_k\sigma(v^{(n)}_l)dx+\int_{\Om} (J^{(n)}\cA^{(n)}-J^{(n-1)}\cA^{(n-1)}) \mS_{\cA^{(n)}}(v^{(n)})\cdot \pa_k\sigma(v^{(n)}_l)dx\\
&+\int_{\Om}  J^{(n-1)}\cA^{(n-1)}\cdot \mS_{\big(\cA^{(n)}-\cA^{(n-1)}\big)}v^{(n)}\cdot \pa_k\sigma(v^{(n)}_l)dx
\end{align*}
Under the assumption $\|v^{(n)}\|_{Y_T}^2\leq M$, we have
\begin{align*}
&-\int_{\Om}\Big(\na_{J^{(n)}\cA^{(n)}}\cdot \mS_{\cA^{(n)}}v^{(n+1)}-\na_{J^{(n-1)}\cA^{(n-1)}}\cdot \mS_{\cA^{(n-1)}}v^{(n)}\Big)~ \sigma(v^{(n)})dx\\
&\geq\,c_0\|\sigma(v^{(n)})\|_{H^1}^2-C_0\|\brho^{\f12}\sigma(v^{(n)})\|_{L^2}^2\\
&\qquad
-\Big|\int_{\Om}(J^{(n)}\cA^{(n)}-J^{(n-1)}\cA^{(n-1)}) \mS_{\cA^{(n)}}(v^{(n)}): \na\sigma(v^{(n)})dx\Big|\\
&\qquad-\Big|\int_{\Om}J^{(n-1)}\cA^{(n-1)}\cdot \mS_{\big(\cA^{(n)}-\cA^{(n-1)}\big)}v^{(n)}:\na\sigma(v^{(n)})dx\Big|\\
&\triangleq\,c_0\| \sigma(v^{(n)})\|_{H^1}^2-C_0\|\brho^{\f12}\sigma(v^{(n)})\|_{L^2}^2-I_1-I_2,
\end{align*}
where we use $|J^{(n)}|\geq \sigma_0$ and Lemma \ref{lem: equal na v} for $\cA^{(n)}.$

For $I_1$, owing to
\begin{align*}
&J^{(n)}\cA^{(n)}-J^{(n-1)}\cA^{(n-1)}=\big(\na (\eta^{(n)}-\eta^{(n-1)})\big)^*\\
&\qquad=\big(\int_0^t \na\sigma(v^{(n-1)})ds \big)^*\sim \big(\int_0^t \na\sigma(v^{(n-1)})ds \big)^2,
\end{align*}
then
\begin{align*}
&\|J^{(n)}\cA^{(n)}-J^{(n-1)}\cA^{(n-1)}\|_{L^2}\\
&\leq C t \|\na \sigma(v^{(n-1)})\|_{L^2_t L^2}\|\frak{ D}(\sigma(v^{(n-1)}))\|_{L^2_t}\leq C t \|\na \sigma(v^{(n-1)})\|_{L^2_t L^2}.
\end{align*}
Applying Holder inequality and Lemma \ref{lem:assumption} to $\cA^{(n)}$, one has
\begin{align*}
I_1\leq& \|J^{(n)}\cA^{(n)}-J^{(n-1)}\cA^{(n-1)}\|_{L^2}\|\cA^{(n)}\|_{L^\infty}\|\na v^{(n)} \|_{L^\infty}\|\na\sigma(v^{(n)})\|_{L^2}\\
\leq&Ct\|\na\sigma(v^{(n-1)})\|_{L_t^2L^2}\frak{ D}(v^{(n)})\|\na\sigma(v^{(n)})\|_{L^2}.
\end{align*}
Similarly, we have
\begin{align*}
I_2\leq& Ct\|\na\sigma(v^{(n-1)})\|_{L_t^2L^2}\frak{ D}(v^{(n)})\|\na\sigma(v^{(n)})\|_{L^2}.
\end{align*}
Combining all above estimates, we obtain
\begin{align*}
&-\int_{\Om}\Big(\na_{J^{(n)}\cA^{(n)}}\cdot \mS_{\cA^{(n)}}v^{(n+1)}-\na_{J^{(n-1)}\cA^{(n-1)}}\cdot \mS_{\cA^{(n-1)}}v^{(n)}\Big)~ \sigma(v^{(n)})dx\\
\geq&\f34 c_0\|\sigma(v^{(n)})\|_{H^1}^2-C_0\|\brho^{\f12}\sigma(v^{(n)})\|_{L^2}^2
-Ct^2\|\na\sigma(v^{(n-1)})\|_{L_t^2L^2}^2\frak{ D}(v^{(n)})^2.
\end{align*}

{\bf{Estimate of pressure term.}}
Integrating by parts and using $\brho|_{x_3=1}=0, ~\sigma(v^{(n)})|_{x_3=0}=0$, we prove that
\begin{align*}
&-\int_{\Om}\Big( \na_{J^{(n)}\cA^{(n)}}\big((J^{(n)})^{-2}\brho^2\big)-\na_{J^{(n-1)}\cA^{(n-1)}}\big((J^{(n-1)})^{-2}\brho^2\big)\Big)
~ \sigma(v^{(n)})dx\\
&=\int_{\Om}\Big( \cA^{(n)}\big((J^{(n)})^{-1}\brho^2\big)-\cA^{(n-1)}\big((J^{(n-1)})^{-1}\brho^2\big)\Big)
: \na \sigma(v^{(n)})dx\\
&=\int_{\Om}\big(\cA^{(n)}-\cA^{(n-1)}\big)(J^{(n)})^{-1}\brho^2
: \na \sigma(v^{(n)})dx\\
&\qquad+\int_{\Om}\cA^{(n-1)}\big((J^{(n)})^{-1}-(J^{(n-1)})^{-1}\big)\brho^2
: \na \sigma(v^{(n)})dx\\
&\leq Ct^{\f12}\|\na \sigma(v^{(n-1)})\|_{L^2_tL^2}\|\na\sigma(v^{(n)})\|_{L^2}.
\end{align*}
Collecting all above estimates together, we finally obtain
\begin{equation}\label{diff-energy-1}
\begin{split}
&\f{d}{dt}\|\brho^{\f12}\sigma(v^{(n)})\|_{L^2}^2+\f{c_0}2\|\na \sigma(v^{(n)})\|_{L^2}^2\\
&\leq Ct\|\na\sigma(v^{(n-1)})\|_{L_t^2L^2}^2(1+\frak{ D}(v^{(n)})^2)+C_0\|\brho^{\f12}\sigma(v^{(n)})\|_{L^2}^2.
\end{split}
\end{equation}
Integrating \eqref{diff-energy-1} in $t \in [0, T]$ and taking $T$ small enough, we have
\begin{equation}\label{diff-energy-2}
\begin{split}
&\sup_{t\in[0,T]}\|\brho^{\f12}\sigma(v^{(n)}(t))\|_{L^2}^2+\f{c_0}2\int_0^T\| \sigma(v^{(n)}(t))\|_{H^1}^2dt\\
&\leq \|\brho^{\f12}\sigma(v^{(n)}(0))\|_{L^2}^2+CT\|\na\sigma(v^{(n-1)})\|_{L_T^2L^2}^2(T+\int_0^T\frak{ D}(v^{(n)})^2dt),
\end{split}
\end{equation}
and then
\begin{align*}
\sup_{t\in[0,T]}&\|\brho^{\f12}\sigma(v^{(n)}(t))\|_{L^2}^2+\f{c_0}2\int_0^T\| \sigma(v^{(n)}(t))\|_{H^1}^2dt\\
\leq& CT(T+M)\|\na\sigma(v^{(n-1)}\|_{L_T^2L^2}^2\leq CT\|\na\sigma(v^{(n-1)})\|_{L_T^2L^2}^2.
\end{align*}
%
By now, we get that when $T$ takes small enough, then we get
\begin{align*}
&\sup_{t\in[0,T]}\|\brho^{\f12}\sigma(v^{(n)}(t))\|_{L^2}^2+\| \sigma(v^{(n)}(t))\|_{L^2_TH^1}^2\\
&\leq \f12(\sup_{t\in[0,T]}\|\brho^{\f12}\sigma(v^{(n-1)}(t))\|_{L^2}^2+\| \sigma(v^{(n-1)}(t))\|_{L^2_TH^1}^2),
 \end{align*}
which completes this Lemma.
\end{proof}

\subsection{Proof of Theorem \ref{thm: main}.}
From Lemma \ref{lem1: v^(n+1)-v^n}, we know $\{v^{(n)}\}_{n=1}^{\infty}$ is Cauchy sequence in the space $\widetilde{Y}_T.$ So as $n\to \infty,$
\begin{align}\label{convergence: strong}
\left\{
\begin{aligned}
&\brho^{\f12}v^{(n)}\to \brho^{\f12}v\qquad\mbox{in}\quad C([0,T],L^2),\\
&v^{(n)}\to v\qquad\mbox{in}\quad L^2([0,T],H^1).
\end{aligned}
\right.
\end{align}
Due to Lemma \ref{lem: regular} that $\|v^{(n)}\|_{Y_T}^2\leq M$ uniformly in $n\geq 1,$ sequence $\{v^{(n)}\}_{n=1}^{\infty}$ have weakly convergent subsequence. Along with strong convergence \eqref{convergence: strong}, we infer that as $n\to 0$
\begin{align*}\label{convergence: weak1}
\left\{
\begin{aligned}
&v^{(n)}\rightharpoonup^* v\qquad\mbox{in}\quad L^\infty([0,T],X^{12}_{\f12}),\\
& v^{(n)}\rightharpoonup  v,\qquad\,\na v^{(n)}\rightharpoonup \na v\qquad\mbox{in}\quad L^2([0,T],X^{12}),\\
&\na v^{(n)}\rightharpoonup^* \na v\qquad\mbox{in}\quad L^\infty([0,T],L^2),\\
&\brho^{\f12}\pa_t v^{(n)}\rightharpoonup \brho^{\f12}\pa_tv\qquad\mbox{in}\quad L^2([0,T],L^2).
\end{aligned}
\right.
\end{align*}
So the function $v$ satisfies equation \eqref{eq:CNS1} in weak sense. On the other hand,  lower semicontinuity gives bound $\|v\|_{Y_T}^2\leq 2M$, and then \eqref{est-v-flow-map} holds. As a result, thanks to Aubin-Lions's lemma \cite{Simon1990}, we get that  $(v,\,\eta)\in\,C([0, T]; X^{12}_{\f12}\cap H^1(\Omega)) \times C([0, T]; \mathcal{F}_{\kappa}(\Omega))$ by using a standard procedure (cf. the proof of Theorem 3.5 in \cite{M-book2002}), which is a strong solution to \eqref{eq:CNS1}. The uniqueness comes from $L^2$ energy estimates  with zero initial data. More precise, let $(\xi_1,\,v_1)$ and $(\xi_2,\,v_2)$ are solutions to \eqref{eq:CNS1} with same initial data. The same process in Lemma \ref{lem1: v^(n+1)-v^n} deduce that
\begin{align*}
\|v_1-v_2\|_{\widetilde{Y}_T}^2\leq \f12 \|v_1-v_2\|_{\widetilde{Y}_T}^2,
\end{align*}
which implies $v_1=v_2$ and then $\xi_1=\xi_2$ on the time interval $[0,T]$. Furthermore, applying \eqref{diff-energy-2} to the system \eqref{eq:CNS1}, we may readily prove that the solution $(v,\,\eta)\in\,C([0, T]; X^{12}_{\f12}\cap H^1(\Omega)) \times C([0, T]; \mathcal{F}_{\kappa}(\Omega))$ depends continuously on the initial data $(v_0,\,\eta_0)\in\,(X^{12}_{\f12}\cap H^1(\Omega)) \times \mathcal{F}_{\kappa}(\Omega)$. This finish the proof of Theorem \ref{thm: main}.
\qed

\section* {Acknowledgments.}
G. Gui is partially supported by NSF of China under Grant 11571279 and 11331005. C. Wang is partially supported by NSF of China under Grant 11701016.  Y. Wang is partially supported by China Postdoctoral Science Foundation 8206200009.

\end{document}